\documentclass[11pt]{amsart}
\usepackage{amsthm}
\usepackage{mathtools,amssymb,latexsym,graphics,enumerate}
\usepackage[mathscr]{eucal}
\usepackage{amsmath,amsfonts,amsthm,amssymb,comment}
\usepackage{color}
\usepackage{hyperref}
\usepackage[left=1in,right=1in,top=1in,bottom=1in]{geometry}

\numberwithin{equation}{section}

\newcommand{\rr}{\mathbb{R}}

\newcommand{\lan}{\langle}
\newcommand{\ran}{\rangle}
\newcommand{\be}{\begin{eqnarray*}}
\newcommand{\bel}{\begin{eqnarray}}
\newcommand{\ee}{\end{eqnarray*}}
\newcommand{\eel}{\end{eqnarray}}
\newcommand{\ba}{\begin{aligned}}
\newcommand{\ea}{\end{aligned}}
\newcommand{\de}{\Delta}
\newcommand{\al}{\alpha}
\newcommand{\na}{\nabla}
\newcommand{\ep}{\epsilon}

\newcommand{\pa}{\partial}
\newcommand{\wh}{\widehat}
\newcommand{\wt}{\widetilde}

\newcommand{\bu}{{\mathbf{u}}}

\newcommand{\myr}[1]{{#1}}
\newcommand{\myb}[1]{\textcolor{blue}{#1}}

\newtheorem{theorem}{Theorem}

\newtheorem{defn}{Definition}
\newtheorem{lem}{Lemma}

\newtheorem{rmk}{Remark}

\newtheorem{example}{Example}

\numberwithin{theorem}{section}
\numberwithin{defn}{section}
\numberwithin{remark}{section}
\numberwithin{example}{section}
\numberwithin{lem}{section}

\newcommand{\norm}[1]{\left\lVert#1\right\rVert}

\newcommand\Torus{{\mathbb T}}
\newcommand\Real{{\mathbb R}}

\newcommand{\dss}{\displaystyle}

\mathtoolsset{showonlyrefs=true}

\title{Stirring Speeds Up Chemical Reaction}
\date{\today}

\author{Siming He} \thanks{simhe@math.duke.edu, Department of Mathematics, Duke University}
\author{Alexander Kiselev}\thanks{kiselev@math.duke.edu, Department of Mathematics, Duke University}

\begin{document}
\begin{abstract}
We consider absorbing chemical reactions in a fluid \myr{flow} modeled by the coupled advection-reaction-diffusion equations. In these systems, the interplay between chemical diffusion and fluid transportation causes the enhanced dissipation phenomenon. We show that the enhanced dissipation time scale, together with the reaction coupling strength, determines the characteristic time scale of the reaction.
\end{abstract}

\maketitle
\section{Introduction}
Consider the advection-reaction-diffusion systems involving two types of chemicals on $ \mathbb{T}^2$
\begin{align}\label{Reaction-diffusion}
\left\{\begin{array}{ccc}\ba
\pa_t n_1+&\mathbf{u}\cdot \na n_1= \nu_1\de n_1-\epsilon n_1 n_{2},\\
\pa_t n_2+&\mathbf{u}\cdot\na n_2= \nu_2\de n_2-\epsilon n_2 n_1,\quad \na\cdot \mathbf{u}=0,\\
n_\al(0&,x,y)=n_{\al;0}(x,y),\quad \al\in\{1,2\}.\ea\end{array}\right.
\end{align}
Here $n_1,n_2$ denote the chemical densities/biological substances and the vector field $\mathbf{u}$ models the underlying fluid \myr{flow}. The parameters $\nu_\al,\,\ep\in(0,1]$ represent the diffusion coefficients and reaction coefficient, respectively. {If the units are non-dimensionalized, then $\nu_1,\,\nu_2$ are the inverse of the P\'{e}clet numbers and $\ep$ is the quotient between the Damk\"{o}hler number and the P\'{e}clet number (see, e.g., \cite{CrimaldiCadwellWeiss08}).}  The domain is normalized so that $\Torus^2=[-1/2,1/2]^2$.

The influence of the fluid flow on reaction rates is of high importance in many applications. Rigorous mathematical analysis of this question to date has been mostly focused on front propagation phenomena and bulk reaction rates in \myr{the single-species} setting, mostly in the context of a single equation with KPP-type, combustion or bistable nonlinearities. We refer to papers \myr{\cite{CKOR,BH,HZ, ABP, BHN-2, CKR00, ConstantinEtAl08, FKR, KS00, KiselevZlatos,ZlaPercol, ZlaSpeedup}} where further references can also be found. It has been established that the flow can have a strong influence on reaction rates, and the extent of the effect depends strongly on the structure and properties of the flow. Here we will work with the system \eqref{Reaction-diffusion} which, in contrast, models two
reacting densities that are not pre-mixed and disappear in reaction process (forming a new compound not tracked by the model). \myr{We are not aware of earlier results on the influence of fluid flows and diffusion on multi-species reaction speed (in the context of models where more than one reacting density function is involved).}

\myr{One motivation for studying the system \eqref{Reaction-diffusion}} is to gain insight into the marine animals' fertilization processes explored in the experimental papers \cite{Riffelletall04}, \cite{RiffellZimmer07}, and \cite{ZimmerRiffell11}. The fertilization cannot proceed unless the sperms and eggs meet. To alleviate the unpredictability of the underlying fluid stream, many marine animals' eggs (e.g., abalones)  emit chemical signals to guide their sperms. Thus the chemotactic attraction between gametes and passive transport by fluid play significant roles in the process. \myr{
In the experiments carried out in \cite{Riffelletall04,RiffellZimmer07,ZimmerRiffell11}, the scientists put the gametes of the abalone in a Taylor-Couette tank and studied the relation between the fertilization success rate and the magnitude of the fluid flow. As a result, a non-trivial connection is discovered between the two quantities. Furthermore, the scientists observe that there exists an optimal shear strength that optimizes the fertilization rate. However, the mathematical understanding of these experiments is lacking.
} Rigorous analysis of the impact of chemotactic attraction was initiated in \cite{KiselevRyzhik12}, \cite{KiselevRyzhik122}, and \cite{KiselevNazarovRyzhikYao20}. The authors proved that chemotactic attraction can significantly decrease the half-life of biological substances in the framework of their models. On the other hand, the role played by the passive fluid transport was investigated in \cite{CrimaldiCadwellWeiss08}, \cite{KiselevRyzhik12}, and \cite{KiselevRyzhik122}.

This paper aims to quantify the relationship between the half-life of the chemical/biological substances and the coefficients $\nu_\al, \, \ep$ involved in \eqref{Reaction-diffusion}. In this paper, we do not consider the chemotactic attraction effects.  
Instead, we focus on strongly mixing flows modeling turbulent regime and shear flows. \myr{Marine scientists believe that these flows play essential roles in various fertilization processes in the ocean (see, e.g., \cite{DennyShibata89} (turbulent regime) and  \cite{Riffelletall04,RiffellZimmer07,ZimmerRiffell11} (shear regime))}.   
We leave the analysis of the complete advection-reaction-diffusion systems subject to chemotactic attraction \myr{for} future work.

If the ambient fluid \myr{flow} $\mathbf{u}$ vanishes, the system \eqref{Reaction-diffusion} has two natural time scales, i.e., the diffusion time scale $O(\min_\al\nu_\al^{-1})$ and the reaction time scale $O(\ep^{-1})$.
The largest of these scales determines the typical time scale of the chemicals/biological substances. To see this, one can consider the initial configuration where the densities $n_{1;0}$ and $n_{2;0}$ are supported away from each other. Then it takes $O(\min_\al\nu_\al^{-1})$ time for the two types of gametes/chemicals to encounter \myr{one another}. Once the densities are mixed, significant reaction occurs on a time scale $O(\ep^{-1})$. 
\myr{To conclude, we come to the heuristic that the net reaction time scale is the sum of the diffusion time scale and the reaction time scale.  }

The system \eqref{Reaction-diffusion} possesses another time scale associated with the non-trivial fluid \myr{flow} $\mathbf{u}$. It is commonly referred to as the `enhanced dissipation time scale' in the fluid mechanics community. 
The enhanced dissipation phenomena naturally arise in the passive scalar equations 
\begin{align}\label{PS_introduction}
\pa_t f +\mathbf{u}\cdot \na f=&\nu \de f, \quad f(t=0,x ,y)=f_{0}(x,y ).
\end{align}
Let us consider \eqref{PS_introduction} in the periodic setting. 
Suppose the diffusion coefficient $\nu$ is small enough and suitable zero average constraints are enforced. In that case, the $L^2$ norm of the solutions to  \eqref{PS_introduction} decays to half of its initial value on a time scale that is much shorter than the diffusion time scale $O(\nu^{-1})$. This fast scale is the enhanced dissipation (time) scale associated with $\mathbf{u}$. \myr{In the two-species reaction model \eqref{Reaction-diffusion}, one expects the two chemical densities to be well-mixed after the enhanced dissipation scale. As a result, introducing ambient fluid flow advection can improve the net reaction time to the sum of the enhanced dissipation scale and the reaction time scale.  
}

We consider two types of vector fields $\bu$ which possess enhanced dissipation, i.e., the relaxation enhancing flows and the shear flows.

P. Constantin et al. \cite{ConstantinEtAl08} 
introduced the notion of relaxation enhancing (R.E.) flows. Under the zero average constraint, \begin{align}\label{average_zero_constraint_2D_flow}
\int_{\mathbb{T}^2} f_{0}  dxdy=0,
\end{align}
the flow $\bu$ is relaxation enhancing if the solutions to \eqref{PS_introduction} have enhanced dissipation phenomenon. In \cite{ConstantinEtAl08}, explicit criterion for the flow to be R.E. is provided. Some examples of R.E. flows are well-known. For instance, the weakly mixing flows are relaxation enhancing, see e.g., \cite{vonNeumann32,Kolmogorov53,Sklover67,  Fayad02, Fayad06}, and the references therein. In works \cite{ElgindiCotiZelatiDelgadino18} and \cite{FengIyer19}, an explicit connection between the mixing property of the fluid flows and the relaxation enhancing property is developed. It is worth noting that explicit constructions of flows with mixing property have attracted much attention, and we refer the interested readers to the works \cite{AlbertiCrippaMazzucato14,AlbertiCrippaMazzucato19, ElgindiZlatos18, YaoZlatos17, BedrossianBlumenthalPunshonSmith191}, and the references therein. Recently, J. Bedrossian et al.  \cite{BedrossianBlumenthalPunshonSmith19} 
showed that certain randomly forced solutions to the Navier-Stokes equations are relaxation enhancing. The result was further applied to derive the Batchelor spectrum in the turbulence theory \cite{BedrossianBlumenthalPunshonSmith192}.

If the vector field $\bu$ is the shear $\bu(x,y)=(u(y),0)$, the enhanced dissipation phenomenon is observed for solutions of \eqref{PS_introduction} subject to zero average constraint \begin{align} \int_\Torus f_0(x,y)dx= 0, \quad \forall y\in \Torus.\label{average_zero_constraint_shear}
\end{align}  In the work \cite{BCZ15}, the authors studied general shear flows' enhanced dissipation effect with the techniques of hypocoercivity \cite{villani2009}. In the paper \cite{Wei18}, the author combined a Gearhart-Pr\"{u}ss type lemma and resolvent estimate to derive the enhanced dissipation of shear flows. Recently,  the authors of \cite{AlbrittonBeekieNovack21} applied the H\"ormander hypoellipticity method to derive the enhanced dissipation estimates in the bounded channel and $\Torus^2$.

The enhanced dissipation phenomena are relevant in other contexts. For example, strong relaxation enhancing flows or shear flows suppress singularity formation in the advective chemotaxis models, see, e.g., \cite{KiselevXu15,BedrossianHe16,He}. Moreover, the enhanced dissipation effect is crucial in understanding the transition threshold in hydrodynamic stability, see, e.g., \cite{BMV14,BGM15I,BGM15II,BGM15III,BVW16, ChenLiWeiZhang18}.  Last but not least,  it was proven in  \cite{Bedrossian17} that a general version of the enhanced dissipation effect suppresses the echo chain instability appeared in nonlinear Landau damping, \cite{MouhotVillani11, BMM13,Bedrossian21}.

Now, we exploit the enhanced dissipation in our analysis of the advective-reaction-diffusion system \eqref{Reaction-diffusion}. Before diving into the details, we introduce some notational conventions.

\noindent
{\bf Notations: } Throughout the paper, the constants $C, C_i\geq 1$, $c\in(0,1)$ are independent of the solutions and the coefficients $\nu,\ep$. The explicit values of $C$'s change from line to line.  The notations $B_{(...)},\, \mathcal{B}_{(...)}$ represent specific bounds/thresholds, whose dependence will be specified. We use $dV$ to denote the volume element, i.e., $dV=dxdy$. The average of the function $f$ on the torus is $\overline{f}=\int_{\Torus^2} fdV$. Functions with subscript $(\cdot)_\sim$, $(\cdot)_{\neq}$ satisfy the zero average constraints \eqref{average_zero_constraint_2D_flow} and \eqref{average_zero_constraint_shear}, respectively.

We organize the main results by distinguishing between the relaxation enhancing flow regime and the shear flow regime. 

\noindent
\textbf{A) Relaxation Enhancing Flow Regime:}
We consider the passive scalar equation \eqref{PS_introduction} subject to the zero average constraint \eqref{average_zero_constraint_2D_flow} which is preserved by the dynamics. Here we provide a quantitative definition of relaxation enhancing flows.
\begin{defn}[$d(\nu)$-relaxation enhancing flows]
The vector field $\mathbf{u}(t,x,y)$ is $d(\nu)$-relaxation enhancing if there exists a threshold $\nu_0(\bu)>0$ such that for $\forall\nu\in(0,\nu_0]$, the solution  $f_{\sim}$ to the passive scalar equation \eqref{PS_introduction} subject to the zero average constraint \eqref{average_zero_constraint_2D_flow} decays as follows:
\begin{align}
||f_{\sim}(s+t)||_2\leq C ||f_{\sim}(s)||_2e^{-\delta  d(\nu)t},\quad \forall t,s\in[0,\infty),\quad \int_{\Torus^2} f_{0;\sim}dxdy =0.\label{Defn:d_nu_RE}
\end{align}The constants $\delta\in (0,1) $ and $C\geq 1$ depend only on the vector field $\bu$ and are independent of $\nu$. The enhanced dissipation rate $d(\nu)$ satisfies the relation $\lim_{\nu\rightarrow 0^+}\frac{\nu}{d(\nu)}=0$.
\end{defn}
We present two examples of relaxation enhancing flows.
\begin{example}[Stochastic Navier-Stokes Flow]
Consider the solutions to the following stochastic Navier-Stokes equations in dimension two\myr{:} 
\begin{align}
\pa_t \mathbf{u}+&(\mathbf{u}\cdot \na )\mathbf{u} +\na p=\de \mathbf{u}+F(t,\omega);\\
\na \cdot \mathbf{u}&=0,\quad \mathbf{u}(t=0,x)=\mathbf{u}_0(x).\label{stochastic_NS}
\end{align}
It was shown in the paper \cite{BedrossianBlumenthalPunshonSmith19} that under specific constraints on the noise, the solutions $\mathbf{u}$ to the equation \eqref{stochastic_NS} are almost surely $|\log\nu|^{-1}$-relaxation enhancing. To be precise, there exist constants $C(\mathbf{u}_0,\omega),\,\delta$, which may only depend on the initial data $\mathbf{u}_0$ and the random realization $\omega$,   such that the solutions to the passive scalar equation \eqref{PS_introduction} subject to the flow $\mathbf{u}$ undergo enhanced dissipation as follows: 
\begin{align}
||f_{\sim} (t)||_{L^2}\leq C (\mathbf{u}_0,\omega)||f_{0;\sim} ||_{L^2}e^{-\delta |\log \nu|^{-1}t},\quad \int_{\mathbb{T}^2} f _{0;\sim}dxdy=0.
\label{PS_ed_2D_mix_introduction}
\end{align}
We refer the interested readers to Theorem 1.3 and Remark 1.4 in the paper \cite{BedrossianBlumenthalPunshonSmith19} for details of the statement.
\end{example}
\begin{example}[Alternating Shear Flow]
Here we introduce another time dependent $\nu^{1/2}$-relaxation enhancing flow. 
We consider the  following alternating shear flow\myr{:}
\begin{align}\label{alternating_shear}
\mathbf{u}( t,x,y)=\sum_{k=0}^\infty&\varphi_{2k}(t)(\sin(2\pi y),0)+\sum_{k=0}^\infty\varphi_{2k+1}(t)(0, \sin(2\pi x)),
\end{align}
\begin{align}
&\varphi_{\ell}(t)=\left\{\begin{array}{ccc}\ba 1,&\quad t\in [(\ell+1/3)K\nu^{-1/2},(\ell+2/3)K \nu^{-1/2}],\\
\mathrm{smooth},&\quad [\ell K  \nu^{-1/2}, (\ell+1/3)K\nu^{-1/2}]\cup [(\ell+2/3)K \nu^{-1/2},(\ell+1)K \nu^{-1/2}),\\
0,&\quad \mathrm{others},\ea\end{array}\right.\\
&\varphi_\ell\in C_c^\infty,\quad \mathrm{support}( \varphi_\ell)\cap \mathrm{support}(\varphi_{\ell+1})= \emptyset,\quad\forall \ell\in\mathbb{N}.
\end{align}
Here $K$ is a universal constant greater than $1$. 
In the appendix, we show that if $\nu^{-1},\,K$ is large enough, the solutions $f_\sim$ to the passive scalar equation  \eqref{PS_introduction} associated with the alternating shear flow decay as follows\myr{:}
\begin{align}
||f_{\sim}  (s+t)||_2\leq 4 ||f_{ \sim} (s)||_2 e^{-\frac{\log 2}{2K}  \nu^{1/2}t},\quad  \int_{\mathbb{T}^2} f_{0;\sim}dxdy=0,\quad \forall s,t\in [0,\infty).\label{ED_alternating_shear}
\end{align}
To conclude, the alternating shear is $\nu^{1/2}$-relaxation enhancing. Moreover, the flow is $C^\infty$ in space and time. \myr{There are several different generalizations.}  The same construction with alternating shear flows in three coordinate directions provides $\nu^{1/2}$-relaxation enhancing flows in $\Torus^3$.  
One can also combine the alternating construction with the rough shear flows in \cite{Wei18,ColomboCotiZelatiWidmayer20} to obtain $|\log \nu|^{-\gamma}$-R.E. flows on $\Torus^2$ for some $\gamma>1$. \myr{We believe that introducing some delicate time-dependent ``phase shifts'' in the construction yields a smooth $\nu^{1/3}$-R.E. flow. In a recent preprint \cite{BlumenthalCotiZelatiGvalani22}, the authors created a smooth $|\log\nu|^{-2}$-R.E. flow by introducing a randomized phase shift into this construction.
}   
\end{example}

With these preparations, we are ready to  state the first main theorem.
\begin{theorem} \label{thm_1}
Consider solutions $n_1,\ n_2$ to the system \eqref{Reaction-diffusion} subject to initial condition $n_{1;0},\ n_{2;0}\in C^2(\Torus^2)$. Assume that the fluid flow is $d(\nu)$-relaxation enhancing. Further assume that $\overline{n_1}(0)\leq \overline{n_2}(0)$. 
\myr{If the total mass of the density $n_1$ is bounded from below on the time interval $[0,T]$, i.e.,
\begin{align}
\inf_{\forall t\in[0,T]}\|n_1(t)\|_{L^1(\Torus^2)}\geq \frac{1}{B}>0,\label{lower_bound}
\end{align}} 
then the following estimate holds on the same time interval\myr{:}
\begin{align}\label{n_1_L_1_decay}
||n_1(t)||_{L^1(\Torus^2)} \leq \frac{4}{3}||n_{1;0}||_{L^1(\Torus^2)}\exp\bigg\{-\frac{1}{C(\bu)B}\bigg(\sum_{\al\in\{1,2\}} d(\nu_\al)^{-1}|\log \nu_\al|+\ep^{-1}\bigg)^{-1}t\bigg\}.
\end{align}
\end{theorem}


\begin{rmk}
The estimates obtained in this paper do not require that the diffusion coefficients $\nu_\al$ are chosen small depending on the initial data, which was always assumed in the other work of enhanced dissipation in nonlinear systems, see, e.g., \cite{BedrossianHe16, CotiZelatiElgindiWidmayer20}. This is due to the fact that the system we consider is dissipative in nature. 
\end{rmk}
\begin{rmk}[Extra logarithmic factor]
The extra $|\log\nu_\al|$ factor is introduced to compensate for various constants appearing during the proof. In particular, when one derives the enhanced dissipation of the solutions in the $L^1$ space, our argument requires a loss in $|\log \nu_\al|$. 
\end{rmk}%

\myr{\begin{rmk}
If we set $B^{-1}=\frac{1}{2}\|n_1(0)\|_{L^1(\Torus^2)}$ in Theorem \ref{thm_1}, then the maximal time interval  $[0,T]$, on which the lower bound \eqref{lower_bound} holds, is commonly referred to as the half-life of the chemical $n_1$. In this case, we can state the decay estimate \eqref{n_1_L_1_decay} purely in terms of the initial data. 
\end{rmk}
}

The result above can be generalized to multi-species absorbing reactions. We consider the systems on $\Torus^2$:
\begin{align}
\pa_t n_\al=\nu_\al\de n_\al -\bu\cdot\na n_\al-\sum_{\beta\in\mathcal I}\ep_{\al\beta}n_\al n_{\beta},\quad n_{\al}(t=0,\cdot)=n_{\al;0}(\cdot),\quad \al, \beta\in\mathcal{I}.\label{system_of_reactions}
\end{align}
Here $\nu_\al>0$ are the diffusion coefficients of the chemicals and $\ep_{\al\beta}\geq 0$ are the reaction coefficients. The total number of chemical species is finite, i.e., $|\mathcal{I}|<\infty$. 
We make the following assumptions
\begin{align}
\min_{\al\in \mathcal{I}}||n_\al(t)||_{L^1}\geq \frac{1}{B_1};\label{Mass_lower_bound}
\end{align}
and
\begin{align}
1\leq\frac{\sum_{\al\in\mathcal{I} } ||n_\al(t)||_1}{\min_{\al\in\mathcal{I}} ||n_\al(t)||_1}\leq B_2.\label{Ratio_upper_bound}
\end{align} Here we note that  the second assumption can be derived from the first one, i.e.,
\begin{align}
1\leq\frac{\sum_{\al } ||n_\al(t)||_1}{\min_\al ||n_\al(t)||_1}\leq B_1{\sum_{\al } ||n_\al(0)||_1} .
\end{align}

\begin{theorem}\label{thm:Systems}
Consider solutions $\{n_\al\}_{\al\in\mathcal{I}}	$ to the system \eqref{system_of_reactions} subject to initial condition $\{n_{\al;0}\}_{\al\in \mathcal{I}}\in C^2(\Torus^2)$.  Assume that the fluid flow is $d(\nu)$-relaxation enhancing. If the assumptions \eqref{Mass_lower_bound}, \eqref{Ratio_upper_bound} hold on the time interval $[0,T]$, then \myr{for all $ t\in[ 0,T]$}, there exist constants $C(B_2), \, C(B_1,B_2,\bu)$ such that 
\begin{align}
\sum_{\al\in \mathcal{I}}||n_\al(t)||_1
\leq C_1(B_2)\sum_{\al\in \mathcal{I}}||n_{\al;0}||_1 \exp\left\{-\frac{1}{C_2(B_1, B_2,\bu)}\bigg(\max_{\al\in\mathcal{I}} \frac{|\log\nu_\al|}{d(\nu_\al)}+(\min_{\al\in \mathcal{I}}\max_{\beta\in \mathcal{I}}\ep_{\al\beta})^{-1}\bigg)^{-1}t\right\}.\label{n_all_L_1_decay}
\end{align}
\end{theorem}

\noindent
\textbf{B) Shear Flow Regime: }
We consider the equation \eqref{Reaction-diffusion} subject to shear flow and diffusion coefficients $\nu=\nu_1=\nu_2$,
\begin{align}\label{Reaction-diffusion-shear}
\left\{\begin{array}{ccc}\ba
\pa_t n_1+&{u}(y)\pa_x n_1= \nu\de n_1-\epsilon n_1 n_{2},\\
\pa_t n_2+&{u}(y)\pa_x n_2= \nu\de n_2-\epsilon n_2 n_1,\\
n_\al(0&,x,y)=n_{\al;0}(x,y),\quad \al\in\{1,2\}.\ea\end{array}\right.
\end{align}
The enhanced dissipation time scale naturally arises in the passive scalar equations subject to shear flow:
\begin{align}
\pa_t f_{\neq}+u(y)\pa_x f_{\neq}=\nu\de f_{\neq},\quad f_{\neq}(t=0,\cdot)=f_{0;\neq}(\cdot), \quad \int_{\mathbb{T}} f_{0;\neq}(x,y)dx=0 \ \text{ for }\forall y\in \mathbb{T}.\label{PS_shear_introduction}
\end{align}
Here the subscript $(\cdot)_{\neq}$  emphasizes that the zero average constraint in \eqref{PS_shear_introduction} is enforced.
The zero average condition rules out the $x$-independent solutions to \eqref{PS_shear_introduction}, for which it turns into heat equation with diffusion coefficient $\nu$.

Now we abuse notation a bit and provide a quantitative definition of  shear flows with enhanced dissipation.
\begin{defn}[$d(\nu)$-relaxation enhancing shear flows]
The shear flow $\mathbf{u}(x,y)=(u(y),0)$ is $d(\nu)$-relaxation enhancing if there exists a threshold $\nu_0(u)>0$, such that for $\forall \nu\in(0,\nu_0]$, all solutions to the passive scalar equation $f_{\neq}$ \eqref{PS_shear_introduction}  decay as follows:
\begin{align}
||f_{\neq}(t)||_2\leq C ||f_{0;\neq}||_2e^{-\delta d(\nu)t},\quad \forall t\in[0,\infty),\quad \int_{\Torus} f_{0;\neq}(x,y)dxdy \equiv0,\quad \forall {y}\in \Torus.\label{Defn:d_nu_shear_RE}
\end{align} Here the constants $\delta\in(0,1)$ and $C$  depend only on the shear profile. The enhanced dissipation rate $d(\nu)$ satisfies the relation $\lim_{\nu\rightarrow 0^+}\frac{\nu}{d(\nu)}=0$.
\end{defn}

It is well-known that the enhanced dissipation rate $d(\nu)$ is closely related to the maximal vanishing order of the shear flow profile $u(y)$, see, e.g., \cite{BCZ15,Wei18,  ElgindiCotiZelatiDelgadino18, CotiZelatiDrivas19,He21, AlbrittonBeekieNovack21}.  Let us consider the shear flow $\mathbf{u}(x,y)=(u(y),0)$ with profile $u(y)$ that has finitely many critical points $\{y_k \}_{k=1}^N$.  We define the \emph{vanishing order} $j(k)$ associated with the critical point $y_k$ as the smallest integer such that
\begin{align}
\frac{d^{\ell}u}{dy^\ell}&(y_k )=0,\quad \frac{d^{j(k)+1}u}{dy^{j(k)+1}}(y_k)\neq 0,\quad \forall \myr{1\leq} \ell \leq j(k).
\end{align}
We further define the \emph{maximal vanishing order}  $j_{m}$ of the shear flow profile $u(y)$ to be  $j_{m}:=\max_{k=1}^N \{j(k)\}$.  Note that any smooth shear flow profile on the torus $\Torus$ must have at least one critical point and hence the maximal vanishing orders $j_{m}$ associated with them are greater than or equal to $1$. 

\ifx
If the shear flow is nondegenerate, J. Bedrossian and M. Coti-Zelati applied hypocoercivity  techniques \cite{villani2009} to show that for $\nu$ small enough, there exist universal constants $C,\delta$ such that the following estimate holds for solutions to the passive scalar equations \eqref{PS_shear_introduction}, (\cite{BCZ15})
\begin{align}
||f_{\neq}(t)||_{L^2}\leq C ||f_{0;\neq}||_{L^2}e^{-\delta  \nu^{1/2}|\log \nu|^{{-2}}t},\quad\forall \nu\in(0,\nu_0(u)], \,\,\,\forall t\geq0.\label{PS_ed_introduction}
\end{align}
D. Wei improved the result in the paper \cite{Wei18} by applying the resolvent estimate and a Gearhart-Pr\"{u}ss type lemma
\begin{align}
||f_{\neq}(t)||_{L^2}\leq C ||f_{0;\neq}||_{L^2} e^{-\delta \nu^{1/2}t},\quad\forall \nu\in(0,\nu_0(u)], \,\,\,\forall t\geq0.
\end{align}
Hence the nondegenerate shear flows are $\nu^{1/2}$-relaxation enhancing. In general, 
\fi

If the shear flow profile has maximal vanishing order $j_{m}$, the above-mentioned works \cite{BCZ15, Wei18,He21} and \cite{AlbrittonBeekieNovack21} provide the following enhanced dissipation estimates for the solutions to \eqref{PS_shear_introduction}
\begin{align}
||f_{\neq}(t)||_{L^2}\leq C ||f_{0;\neq}||_{L^2}e^{-\delta \nu^{\frac{j_{m}+1}{j_{m}+3}}t},\quad\forall \nu\in(0,\nu_0(u)], \,\,\,\forall t\geq0.\label{ED_shear_introduction} 
\end{align}
As a result, we see that the shear flow is $\nu^{\frac{j_{m}+1}{j_{m}+3}}$-relaxation enhancing.
In the paper \cite{CotiZelatiDrivas19}, M. Coti-Zelati and D. Drivas showed that 
this $d(\nu)$-enhanced dissipation rate is sharp.

To present the main theorem, we introduce the notions of the $x$-average and the remainder:
\begin{align}\label{x_average_and_remainder}
\lan f\ran(y)=\int_{\Torus} f(x,y)dx,\quad f_{\neq}(x,y)=f(x,y)-\lan f\ran(y).
\end{align}
Our main result in the shear flow regime is as follows.
\begin{theorem}\label{thm_3}
Consider the solutions $n_1, n_2$ to \eqref{Reaction-diffusion} subject to initial condition $n_{1;0},\ n_{2;0}\in C^2(\Torus^2)$. Assume that the shear  flow is $d(\nu)$-relaxation enhancing with decay rate $\delta d(\nu)$ and threshold $\nu_0$. Moreover, suppose that $||n_{1;0}||_{L^1}\leq ||n_{2;0}||_{L^1}$. 
 Then for $0<\nu\leq \nu_0$, there exist two characteristic times
 \begin{align}\label{t_0}
\mathcal{T}_0=\frac{1}{\delta  d(\nu)}\log \left(C\bigg(\frac{\epsilon}{\delta d(\nu)}\sum_\al||n_{\al;0}||_{L_{x,y}^2}+1\bigg)\frac{\sum_\al||n_{\al;0}||_{L_{x,y}^2}}{||\min_\al \lan n_{\al;0}\ran||_{L^1_y}}\right). 
\end{align}
and \begin{align}\label{t_1}
\mathcal{T}_1={\ep^{-1}C\max\bigg\{1,\norm{\min_{\al\in\{1,2\}}\lan n_{\al;0}\ran}_{L^1_y}^{-1}}\bigg\}=:\ep^{-1}\mathcal{B}_1
\end{align} such that significant mass is consumed by the time $\mathcal{T}_0+\mathcal{T}_1$:
\begin{align}\label{comsumption_est_thm_3}
||n_{1;0}||_{L^{1}(\Torus^2)}-||n_1(\mathcal{T}_0+\mathcal{T}_1)||_{L^1(\Torus^2)}\geq \frac{1}{12}\norm{\min_{\al\in\{1,2\}}\lan n_{\al;0 }\ran}_{L_y^1}. 
\end{align}
\end{theorem}
\begin{rmk}
Modulo logarithmic factors, the time \myr{$\mathcal T_0$} is of order $O(d(\nu)^{-1})$, which is the enhanced dissipation time scale. If the `overlapping mass' $\norm{\min_{\al\in\{1,2\}}\lan n_\al\ran (0 )}_{L_y^1}$ is not too small, the time \myr{$\mathcal T_1$} is of order $O(\ep^{-1})$. As a result, we observe that the total reaction time is determined by the larger one of the reaction time scale and the enhanced dissipation time scale.
\end{rmk}
\begin{rmk}The main difficulty in extending the above result to systems with different diffusion coefficients is that one of the key lemmas, i.e., Lemma \ref{lem:wt n_decay}, does not hold in general. As a result, keeping track of the time evolution of the $\min_\al \lan n_\al\ran$ becomes challenging.
\end{rmk}

\myr{\begin{rmk}
We believe the theorem can be extended to include chemical reactions on the plane subject to the two-dimensional vortices, as in \cite{CrimaldiCadwellWeiss08}. Here, one can apply the relaxation enhancing estimates for vortices in \cite{CotiZelatiDolce20}. We will leave this problem to future work. 
\end{rmk}}

The paper is organized as follows: in Section \ref{Sect:thm_1}, we prove Theorem \ref{thm_1}; in Section \ref{Sect:thm_2}, we prove Theorem \ref{thm:Systems}; in Section \ref{Sec:thm_3}, we prove Theorem \ref{thm_3}; in appendix, we prove the enhanced dissipation of the alternating flows. 

\section{Proof of Theorem \ref{thm_1}}\label{Sect:thm_1}


First of all, we apply the Nash inequality to prove the following lemma which provides $L^1$-estimates for the passive scalar solutions.
\begin{lem}[$L^1$-decay of the passive scalar solution]\label{Lem:L1}
Consider solutions $\eta_{\sim}$ to the passive scalar equation \eqref{PS_introduction} subject to zero average constraint, i.e., $\int_{\Torus^2} \eta_{\sim}(x,y)dV=0$. Assume that the flow $\mathbf{u}$ is $d(\nu)$-relaxation enhancing with decay rate $\delta d(\nu)$. Then if $\nu$ is small enough, there exist constants $c\in(0,1),\,C$ such that the following estimate holds
 \begin{align}\label{L1_ED}
|| \eta_{\sim}(s+t)||_1\leq C ||\eta_{\sim}(s)||_1 e^{-c\delta d(\nu)|\log\nu|^{-1}t},\quad \forall s,t\in[0,\infty).
\end{align}
\end{lem}
\begin{proof}
We begin with the derivation of  $L^1$-$L^2$-estimate of the passive scalar semigroup. Consider the time interval $[s, s+4\delta^{-1}d(\nu)^{-1}|\log \nu|]$. 
We estimate the time evolution of the $L^2$-norm using Nash inequality as follows:
\begin{align}
\frac{d}{dt}\frac{1}{2}||\eta_{\sim}(s+t)||_2^2\leq& -\nu ||\na \eta_{\sim}(s+t)||_2^2
\leq-\frac{\nu||\eta_{\sim}(s+t)||_2^4}{C_N||\eta_{\sim}(s+t)||_1^2}\leq-\frac{\nu||\eta_{\sim}(s+t)||_2^4}{C_N||\eta_{\sim}(s)||_1^2}.
\end{align}
Here the $L^1$-norm of $\eta_{\sim}$ is non-increasing because we can consider the solutions to \eqref{PS_introduction} \myr{evolving from} the positive and negative part of the initial data, i.e., $ \eta_{\sim}^\pm(s,x,y)=\max\{\pm \eta_{\sim}(s,x,y),0\}.$ Since both of them are positive and have conserved $L^1$-norms and $\eta_{\sim}$ is the sum of these two solutions, we have that the $L^1$-norm of $\eta_{\sim}(s+t)$ does not exceed the \myr{$L^1$-norm of $ \eta_{\sim}(s)$}. Next we directly solve the ordinary differential inequality 
subject to arbitrary positive initial data 
and obtain that there exists a universal constant $C$ such that the following estimate holds
\begin{align}||\eta_{\sim}(s+t)||_2\leq\frac{C}{(\nu t)^{1/2}}||\eta_{\sim}(s)||_1.\label{Nash}
\end{align}
Now we decompose the interval $[s, s+4\delta^{-1}d(\nu)^{-1}|\log \nu|]$ into two sub-intervals and apply the following estimate:
\begin{align*}
||\eta_{\sim}&(s+4\delta^{-1}d(\nu)^{-1}|\log\nu|)||_{1}\\
\leq & ||\eta_{\sim}(s+4\delta^{-1}d(\nu)^{-1}|\log\nu||)||_2
\leq C||\eta_{\sim}(s+  \delta ^{-1}d(\nu)^{-1}|\log \nu|)||_2 {e^{-\delta  d(\nu)3\delta ^{-1}d(\nu)^{-1}|\log \nu|}}\\
\leq &Ce^{-|\log\nu^3|}||\eta_\sim(s+\delta^{-1}d(\nu)^{-1}|\log\nu|)||_2\\
\leq& \frac{C\nu^{3}}{(\nu\delta^{-1}d(\nu)^{-1}|\log \nu|)^{1/2}}||\eta_{\sim}(s)||_{1}\leq \frac{1}{2}||\eta_{\sim}(s)||_1.
\end{align*}
In the last line, we choose $\nu$ small enough compared to universal constants so that the coefficient is small. We further note that the $L^1$-norm of $\eta_{\sim}$ is non-increasing along the dynamics. To conclude, we iterate the argument on consecutive intervals to derive the estimate \eqref{L1_ED}.
\end{proof}

\ifx
\begin{align}
||S&_{s,s+c^{-1}\delta^{-1}d(\nu)^{-1}|\log \nu|}\eta  _{\sim}(s)||_1\\
=&C||S_{s+\frac{1}{2}c^{-1}\delta ^{-1}d(\nu)^{-1}|\log \nu|,s+c^{-1}\delta ^{-1}d(\nu)^{-1}|\log \nu|}S_{s,s+\frac{1}{2}c^{-1}\delta ^{-1}d(\nu)^{-1}|\log \nu|}\eta_{\sim}(s)||_2\\
\leq&C||\eta_{\sim}(s+ \frac{1}{2}c^{-1}\delta ^{-1}d(\nu)^{-1}|\log \nu|)||_2 {e^{-\frac{1}{2}\delta  d(\nu)c^{-1}\delta ^{-1}d(\nu)^{-1}|\log \nu|}}\\
\leq&\frac{C }{(\nu \frac{1}{2}c^{-1}\delta ^{-1}d(\nu)^{-1}|\log \nu|)^{1/2}}{e^{-\frac{1}{2}\delta  d(\nu)c^{-1}\delta ^{-1}d(\nu)^{-1}|\log \nu|}}||\eta_{\sim}(s)||_1
\leq \frac{1}{64}||\eta_{\sim}(s)||_1.&\label{L2_Linfty_semigroup}
\end{align} 
\fi
\ifx
Next, we will assume that the $L^\infty$ norm of the remainder $q_{\sim}$ is bounded by a small constant. The reasoning is as follows. At the time $t_{\mathrm{in}}$, the $L^2$ norm of the solution $||q_{\sim}(t_{\mathrm{in}})||_2\leq C_q=\mathcal{O}(1)$. By applying the Nash inequality and duality argument, we obtain that there exists a universal constant $C$ such that the solution to the passive scalar equation has the following decay in dimension two:
\begin{align}\label{Nash}
||q_{\sim}(t+s)||_\infty&\leq C (\nu t)^{-1/2}||q_{\sim}(s)||_2, \quad \forall s\geq 0, t\geq 0.
\end{align} Applying the semigroup estimate \eqref{Semigroup_est_stochastic_flow} yields that
\begin{align}
||q_{\sim}(t_{\mathrm{in}}+s+\log\nu^{-1})||_\infty&\leq C (\nu \log\nu^{-1})^{-1/2}e^{-\delta_{ED} \log^{-1}\nu^{-1} s}||q_{\sim}(t_{\mathrm{in}})||_2.
\end{align}
By setting $s=2\delta_{ED}^{-1}\log^2\nu^{-1}$ and choose $\nu$ small but independent of the solutions and $\ep$, we have that
\begin{align}
||q_{\sim}(t)||_\infty\leq \frac{1}{32}, \quad \forall t\geq t_\mathrm{in}+\log\nu^{-1}+2\delta_{ED}^{-1}\log^2\nu^{-1}.
\end{align}
With this estimate, one can prove the enhanced decay in \eqref{slow_regime_p_neq_decay_2D_flow} starting at time $t_\mathrm{in}+\log\nu^{-1}+2\delta_{ED}^{-1}\log^2\nu^{-1}$. On the time interval $[t_{\mathrm{in}},t_{\mathrm{in}}+\log\nu^{-1}+2\delta_{ED}^{-1}\log^2\nu^{-1}]$, we use the simple estimate \eqref{simple_est_of_p_neq}. Since this interval have length of order $\delta_{ED}^{-1}\log^2\nu^{-1}$, the exponential decaying factor in \eqref{slow_regime_p_neq_decay_2D_flow} is bounded below by a constant. Hence by applying a slightly larger constant $C_{ED}$, we prove the estimate \eqref{slow_regime_p_neq_decay_2D_flow} on this time interval.
\fi

\begin{proof}[Proof of Theorem \ref{thm_1}]
We organize the proof in three steps.

\noindent
\textbf{Step \# 1: Preparations. } In this step, we translate the continuous-in-time  decay estimate \eqref{n_1_L_1_decay} into a discrete-in-time one and properly decompose the time horizon.

We recall that the average density $\overline{n_1}$ is bounded from below by $1/B$ on the maximal time interval $[0,T]$.  Fix an arbitrary instance $t_0$ in  $[0,T]$ and define the net reaction time as $T_\star:=C(B)(\sum_{\al\in\{1,2\}}\delta^{-1}d(\nu_\al)^{-1}|\log \nu_\al|+\ep^{-1})$. The constant $C(B)$, which will be chosen later in the proof, depends on the mass threshold $1/B$ and the constants $C, c$ in Lemma \ref{Lem:L1}. The estimate \eqref{n_1_L_1_decay} is ensured if we can show that the total mass $||n_1||_{L^1}$ decays by a fixed proportion by the time $t_0+T_\star$, i.e.,
\begin{align}\label{goal}
||n_1(t_0+T_\star)||_{L^1}\leq \frac{3}{4}||n_1(t_0)||_{L^1},\quad [t_0,t_0+T_\star]\subset[0,T].
\end{align}
To see this implication, we pick a time $t\in [0,T]$ and determine the largest integer $m\in \mathbb{N}$ such that $m T_\star\leq t$. The choice of $m$ guarantees the relation $t\leq (m+1)T_\star$. Then invoking the estimate \eqref{goal} yields that
\begin{align}
||n_1(t)||_{L^1}=||n_1(t-m T_\star+mT_\star)||_{L^1}\leq ||n_1(t-mT_\star)||_{L^1}\left(\frac{3}{4}\right)^m\leq ||n_1(t-mT_\star)||_{L^1} e^{-(\frac{t}{T_\star}-1)\log \frac{4}{3}}.
\end{align}
Direct calculation yields that the $L^1$-norm $||n_1(t)||_1$ is decreasing in time. Hence,
\begin{align}
||n_1(t)||_{L^1}\leq\frac{4}{3} ||n_{1;0}||_{L^1}e^{-\frac{t}{T_\star}\log \frac{4}{3}}=\frac{4}{3} ||n_{1;0}||_{L^1}\exp\bigg\{-\frac{C^{-1}(B)\log\frac{4}{3}}{\sum_{\al=1}^{2}\delta^{-1}d(\nu_\al)^{-1}|\log \nu_\al|+\ep^{-1}} t\bigg\}.
\end{align}
Modulo small adjustment to constants, this is the result \eqref{n_1_L_1_decay}.

Next we introduce the partition of time interval $[t_0,t_0+T_\star]$. One of the obstacles to proving \eqref{goal} is that two distinct phenomena occur on the interval $[t_0,t_0+T_\star]$, with enhanced dissipation and chemical reaction involved. 
Hence our strategy is to decompose the time interval into two parts and focus on deriving the enhanced dissipation estimates on the first part and the reaction estimates on the second. To be precise, we define
\begin{align}
[t_0,t_0+T_\star]=&[t_0,t_0+T_1)\cup [t_0+T_1,t_0+T_1+T_2], \\
T_1:=&C_1\sum_{\al=1}^2 \delta^{-1}d(\nu_\al)^{-1}|\log \nu_\al|, \quad T_2:=16B\ep^{-1}\log 2.\label{T_1_T_2_thm_1}
\end{align}
Here the universal constant $C_1$ depends only on the constants $c,\,C$  appeared in \eqref{L1_ED} and will be chosen in \eqref{choice_of_C_1}. Since $|\Torus|=1$, the estimate \eqref{goal} is equivalent to 
\begin{align}
\overline{n_1}(t_0+T_1+T_2)\leq \frac{3}{4}\overline{n_1}(t_0).
\end{align}
This concludes step \# 1.

\noindent
\textbf{Step \# 2: Nonlinear enhanced dissipation estimates.}
To derive \eqref{goal}, the first main estimate we require is the nonlinear enhanced dissipation estimate at time instance $t_0+T_1$. The challenge is that the reaction coefficient $\ep$ can be much larger than $d(\nu_\al)$, and the nonlinear term cannot be treated perturbatively in general. Our idea is that on the time interval $[t_0,t_0+T_1)$, one considers the super solutions
\begin{align}
\pa_t \wt n_\al+\bu\cdot\na \wt n_\al=\nu_\al\de \wt n_\al ,\quad \wt n_\al(t_0,\cdot)=n_\al(t_0,\cdot),\quad \al\in\{1,2\},
\end{align}
and uses the total reacted mass
\begin{align}
Q(t):=\ep\int_0^t  \int_{\Torus^2} n_1n_2dVds\label{Q}
\end{align}
to control the deviation between the  super solutions and the real solutions. The same quantity $Q$ is considered in the paper \cite{CrimaldiCadwellWeiss08}.
Direct calculations yield the following relation
\begin{align}\label{Q_difference}&&\quad
Q(t_0+t)-Q(t_0)=&\ep\int_{t_0}^{t_0+t} \int_{\Torus^2} n_1n_2dVds
=||n_\al (t_0)||_1-||n_\al(t_0+t)||_1=||\wt n_\al-n_\al||_1(t_0+t).
\end{align}
{Explicit justification of \eqref{Q_difference} is as follows. First of all, by integrating the equation \eqref{Reaction-diffusion} in space and time, one obtains the relation $||n_\al(t_0)||_{L^1}-||n_\al(t_0+t)||_{L^1}=Q(t_0+t)-Q(t_0)$. Next we observe that since $\wt n_\al$ are super-solutions, the differences $\wt n_\al-n_\al$ are greater than zero. Hence integrating the equations of $\wt n_\al-n_\al$ yields the last equality in \eqref{Q_difference}.}

With the total reacted mass $Q$ introduced, we are ready to derive the nonlinear enhanced  dissipation estimate.
By the linear enhanced dissipation estimate \eqref{L1_ED} and the fact that $\wt n_{\al;\sim}$ solves the passive scalar equation subject to zero average constraint \eqref{average_zero_constraint_2D_flow}, we can choose the $C_1$ in \eqref{T_1_T_2_thm_1} large enough such that \myr{for all $s\geq 0$}
\begin{align}
||\wt n_{\al;\sim}(t_0+T_1+s)||_1\leq \frac{1}{16}||\wt n_{\al;\sim}(t_0)||_1\leq\frac{1}{8}||n_\al(t_0)||_1=\frac{1}{8}\overline{n_\al}(t_0),\quad\al\in\{1,2\}. \label{choice_of_C_1}
\end{align}
Hence, the $L^1$-norm of the remainder $n_{\al;\sim}$ is bounded, i.e.,
\begin{align}
||n_{\al;\sim}(t_0+T_1+s)||_1\leq& ||\wt n_{\al;\sim}(t_0+T_1+s)||_1+||\wt n_{\al;\sim}(t_0+T_1+s)-n_{\al;\sim}(t_0+T_1+s)||_1\\
\leq& ||\wt n_{\al;\sim}(t_0+T_1+s)||_1+||\wt n_{\al}(t_0+T_1+s)-n_{\al}(t_0+T_1+s)||_1\\
&+|\overline{\wt n_{\al}(t_0+T_1+s)}-\overline{n_\al(t_0+T_1+s)}|\\
\leq&\frac{1}{8} \overline{n_\al}(t_0)+2Q(t_0+T_1+s)-2Q(t_0),\quad\forall s\geq 0, \, \al\in\{1,2\}.\label{n_al_sim_est}
\end{align}
This concludes the proof of the nonlinear enhanced dissipation estimates and step \# 2.
\ifx we need  \begin{align}
||n_{1;\sim}(t_0+T_1+s)||_1\leq& ||\wt n_{1;\sim}(t_0+T_1+s)||_1+||\wt n_{1;\sim}(t_0+T_1+s)-n_{1;\sim}(t_0+T_1+s)||_1\\
\leq& ||\wt n_{1;\sim}(t_0+T_1+s)||_1+||\wt n_{1}(t_0+T_1+s)-n_{1}(t_0+T_1+s)||_1\\
&+||\overline{\wt n_{1}(t_0+T_1+s)}-\overline{n_1(t_0+T_1+s)}||_1\\
\leq&\frac{1}{16} |\Torus|^2\overline{n_1}(t_0)+2Q(t_0+T_1+s)-2Q(t_0).\label{n_1_sim_est}
\end{align}
\begin{align}
||n_{2;\sim}(t_0+T_1+s)||_1\leq& ||\wt n_{2;\sim}(t_0+T_1+s)||_1+||\wt n_{2;\sim}(t_0+T_1+s)-n_{2;\sim}(t_0+T_1+s)||_1\\
\leq&\frac{1}{16} |\Torus|^2\overline{n_2}(t_0)+2Q(t_0+T_1+s)-2Q(t_0).\label{n_2_sim_est}
\end{align}
\fi

\noindent
\textbf{Step \# 3: Proof of estimate \eqref{goal}.}
We focus on the second time component $[t_0+T_1,t_0+T_1+T_2]$ and distinguish between two possible cases.

\noindent \textbf{Case a)}
If $Q(t_0+T_1+s)-Q(t_0)=||n_1(t_0)||_1-||n_1(t_0+T_1+s)||_1\geq \frac{1}{4}\overline{n_1}(t_0)$ for some $0\leq s<T_2$, then positivity and the fact that $\overline{n_1}(t)$ is decreasing in time yields that
\begin{align}
\overline{n_1}(t_0+T_1+T_2)\leq \frac{3}{4}\overline{n_1}(t_0).
\end{align}
Thus we have the estimate \eqref{goal}.

\noindent \textbf{Case b)}
The negation to the condition in case a) is that
\begin{align}\label{Assumption_Q}
Q	(t_0+T_1+s)-Q(t_0)< \frac{1}{4}\overline{n_1}(t_0),\quad\forall s\in[0, T_2) .
\end{align}
To establish \eqref{goal}, we make three preparations and estimate the time evolution of $\overline{n_1}$. Combining \eqref{Assumption_Q} and the relation \eqref{Q_difference} yields   that
\begin{align}\label{Assumption_Q_2}
& &\ \ \,\frac{1}{4}\overline{n_1}(t_0)>||n_\al(t_0)||_1-||n_\al(t_0+T_1+s)||_1
=\overline{n_\al}(t_0)-\overline{n_\al}(t_0+T_1+s),\quad \forall s\in[0,T_2), \al=1,2.
\end{align}
Hence,
\begin{align}\label{Assumption_overline_n_1}
\overline{n_2}(t_0+T_1+s)\geq\overline{n_1}(t_0+T_1+s)> \frac{3}{4}\overline{n_1}(t_0),\, \forall s\in[0,T_2).
\end{align}
Next, we apply positivity of {$n_\al(x,y)=\overline{n_\al}-n_{\al;\sim}^-(x,y)\geq 0,\quad \forall (x,y)\in\{n_{\al;\sim}^->0\}$} to derive that
\begin{align}||n_{\al;\sim}^-||_\infty\leq \overline{n_{\al}},\quad\al\in\{1,2\}.\label{negative_part_bound}
\end{align} For the positive part of the remainder, we apply relation \eqref{n_al_sim_est} and assumption \eqref{Assumption_Q} to obtain
\begin{align}
||n_{\al;\sim}^+(t_0+T_1+s)||_1 =&\frac{1}{2}||n_{\al;\sim}(t_0+T_1+s)||_1\leq\frac{1}{16}\overline{n_\al}(t_0)+Q(t_0+T_1+s)-Q(t_0)\\
\leq &\frac{1}{16}\overline{n_\al}(t_0)+\frac{1}{4}\overline{n_1}(t_0),\quad \al=1,2,\, \forall s\in[0,T_2).\label{positive_part_L1}
\end{align}
Now we consider the time evolution of $\overline{n_1}$\ \
\begin{align}
&\frac{d}{dt}\overline{n_1}(t)=-\ep \overline{n_1}(t)\overline{n_2}(t) -\ep \overline{n_{1;\sim} n_{2;\sim}}(t)
\leq -\ep \overline{n_1}(t)\overline{n_2}(t)+\ep \overline{ n_{1;\sim}^+ n_{2;\sim}^-}(t)+\ep \overline{ n_{1;\sim}^- n_{2;\sim}^+}(t).&
\end{align}
For $t= t_0+T_1+s$, we first apply \eqref{negative_part_bound} to obtain 
\begin{align}
\frac{d}{ds}\overline{n_1}( t_0+T_1+s)\leq&-\ep \overline{n_1}( t_0+T_1+s)\,\overline{n_2}( t_0+T_1+s)+\ep||n_{1;\sim}^+( t_0+T_1+s)||_1\overline{n_2}( t_0+T_1+s)\\
&+\ep||n_{2;\sim}^+( t_0+T_1+s)||_1\overline{n_1}( t_0+T_1+s).
\end{align}
Then by \eqref{Assumption_overline_n_1}, \eqref{positive_part_L1},
\begin{align}
\frac{d}{ds}\overline{n_1}( t_0+T_1+s)\leq&-\ep \overline{n_1}( t_0+T_1+s) \overline{n_2}( t_0+T_1+s)+\ep \frac{3}{8}\overline{n_{1}}(t_0)\overline{n_2}( t_0+T_1+s)\\
&+\ep\left(\frac{1}{4}\overline{n_{1}}(t_0)+\frac{1}{16}\overline{n_2}(t_0)\right)\overline{n_1}( t_0+T_1+s)\\
\leq&-\ep \overline{n_2}( t_0+T_1+s) \left(\frac{1}{2}\overline{n_1}( t_0+T_1+s)- \frac{3}{8}\overline{n_1}(t_0)\right)  \\
&- \ep \overline{n_1}( t_0+T_1+s) \left(\frac{1}{2}\overline{n_2}( t_0+T_1+s)- \frac{1}{4}\overline{n_1}(t_0)-\frac{1}{16} \overline{n_2}(t_0)\right)=:\mathbb{P}_1+\mathbb{P}_2.
\end{align} The relations \eqref{Assumption_Q_2}, \eqref{Assumption_overline_n_1} yield that the first part $\mathbb{P}_1$ is negative and the second term is bounded above:
\begin{align}
\mathbb{P}_2\leq& - \ep \overline{n_1}( t_0+T_1+s) \left(\bigg(\frac{3}{4}\frac{7}{16}-\frac{1}{4}\bigg)\overline{n_1}(t_0)-\frac{1}{16}(\overline{n_2}(t_0)-\overline{n_2}(t_0+T_1+s))\right)\\
\leq& -\frac{1}{16}\ep \overline{n_1}(t_0)  \overline{n_1}( t_0+T_1+s) .
\end{align}  We apply the assumption $\overline{n_1}(t_0)\geq 1/B$ to obtain, \begin{align}
\frac{d}{ds}\overline{n_1}( t_0+T_1+s)
\leq&-\ep\frac{1}{16 B}\overline{n_1}( t_0+T_1+s),.
\end{align}
Now we see that after time $T_2={16}B\ep^{-1}\log 2$, the $\overline{n_1}$ decays to $\frac{3}{4}
$ of its starting value. Hence we prove the estimate \eqref{goal} in case b) and conclude the proof of Theorem \ref{thm_1}.
\end{proof}
\ifx Can we use cutoff to localize the data in the $y$ direction in the shear case? Note that
\begin{align}
\lan \wt n_1-n_1\ran(y)=\sum_{i=-\infty}^\infty\int_0^t\int_\Real\frac{1}{\sqrt{4\pi \nu_1(t-s)}}e^{-\frac{|z-y|^2}{4\nu_1 (t-s)}}\ep \lan n_1n_2\ran_i(z)dzds.
\end{align}
Now we note that the $L^\infty$ norm on the left is determined by the $L_t^{..}L_y^1$-norm of $\lan n_1n_2\ran$. Energy estimate of $L^1$-type?\fi

\textcolor{red}{}
\section{Proof of Theorem \ref{thm:Systems}}\label{Sect:thm_2}
Since the proof is similar to the one in Theorem \ref{thm_1}, we only highlight the main differences.

We consider the total mass $M_{all}$
\begin{align}
M_{all}(t):=\sum_{\al\in\mathcal{I}} ||n_\al(t)||_1.
\end{align}
The positivity of the reaction coefficients $\ep_{\al\beta}\geq 0$  yields that $M_{all}(t)$ is \myr{monotonically }decreasing.
Furthermore, we consider the characteristic reaction time \begin{align}
T_\star=\log(32CB_2) (c\delta)^{-1}\max_{\al\in \mathcal{I}} d(\nu_\al)^{-1}|\log \nu_\al|+2 B_1 \left(\min_{\al\in\mathcal{I}}\max_{\beta\in\mathcal{I}} \ep_{\al\beta}\right)^{-1} \log\left(1-\frac{1}{8B_2}\right)^{-1}=:T_1+T_2.
\end{align}
Here $C,\,c,\, \delta$ are the constants \myr{appearing} in Lemma \ref{Lem:L1}. Recall that $[0,T]$ is the maximal interval on which \eqref{Mass_lower_bound} and \eqref{Ratio_upper_bound} hold. By the argument in step \# 1 within the proof of Theorem \ref{thm_1}, the estimate \eqref{n_all_L_1_decay} is a consequence of the following statement:
\begin{align}\label{goal_system}
M_{all}(t_0+T_\star)\leq  \bigg(1-\frac{1}{8 B_2}\bigg) M_{all}(t_0),
\end{align} where $ [t_0,t_0+T_\star]$ is an arbitrary interval embedded in $[0,T]$.

To acquire the nonlinear enhanced dissipation estimate, we consider the total reacted mass $Q_{all}$,
\begin{align}
\quad Q_{all}(t):=\sum_{\al,\beta\in\mathcal{I}}\int_0^t\int _{\Torus^2}\ep_{\al\beta}n_{\al}n_\beta dV ds.
\end{align} Since $\ep_{\al\beta}\geq 0$,  $Q_{all}(t)$ is non-negative and increasing.
Direct time integration yields that for $\forall t_0,t\in[0,\infty)$
\begin{align}\label{Relation_system}
Q_{all}(t_0+t)-Q_{all}(t_0)
=M_{all}(t_0)-M_{all}(t_0+t)=&\sum_{\al\in \mathcal{I}}||\widetilde{n}_\al(t_0+t)-n_{\al}(t_0+t)||_{L^1},
\end{align}
where $\{\wt{n}_\al\}_{\al\in \mathcal{I}}$ are super solutions to $\{n_\al\}_{\al\in \mathcal{I}}$ defined by
\begin{align}
\pa_t \wt n _\al=\nu_\al \de \wt n_\al -\bu\cdot \na \wt n_\al,\quad \wt n_{\al}(t_0)=n_\al(t_0).
\end{align}
Note that the lower bound on $Q$,
\begin{align}
Q_{all}(t_0+T_\star)-Q_{all}(t_0)\geq \frac{1}{8 B_2}M_{all}(t_0), 
\end{align}
when combined with \eqref{Relation_system}, yields the final result \eqref{goal_system}. Hence we make the following assumption throughout the remaining part of the proof
\begin{align}
Q_{all}(t_0+t)-Q_{all}(t_0)< \frac{1}{8B_2}M_{all}(t_0),\quad \forall t\in[0,T_\star).\label{Assumption_system}
\end{align}

To obtain the nonlinear enhanced dissipation estimate, we invoke \eqref{L1_ED} to derive  the following 
\begin{align}
\sum_\al||\wt n_{\al;\sim}(t_0+T_1+s)||_{L^1}\leq \frac{1}{32B_2}\sum_\al ||n_{\al;\sim}(t_0)||_{L^1}\leq \frac{1}{16B_2}\sum_\al|| n_\al(t_0)||_{L^1} ,\quad \forall s\in[0,T_2].\label{supersolution_remainder}
\end{align}
Application of the relations  \eqref{Relation_system}, \eqref{Assumption_system} and \eqref{supersolution_remainder} then yields that
\begin{align}
||n_{\al;\sim}(t_0+T_1+s)||_{L^1}\leq& ||\wt n_{\al;\sim}(t_0+T_1+s)||_{L^1}+||n_{\al;\sim}(t_0+T_1+s)-\wt n_{\al;\sim}(t_0+T_1+s)||_{L^1}\\
\leq &\frac{3}{8B_2}M_{all}(t_0),\quad \forall s\in[0,T_2].\label{n_al_sim_t_0+T_1+s_L_1}
\end{align}
This is the enhanced dissipation we use in the sequel.

Next we invoke the relations \eqref{Relation_system}, \eqref{Assumption_system}, the assumption $\eqref{Ratio_upper_bound}_{t=t_0}$ and the enhanced dissipation estimate \eqref{n_al_sim_t_0+T_1+s_L_1} to obtain that
\begin{align}
\overline{n}_\al &(t_0+T_1+s)\geq\overline{n}_\al(t_0)-|\overline{n}_\al(t_0)-\overline{n}_\al(t_0+T_1+s)|\\
\geq& \frac{1}{B_2 }M_{all}(t_0)-\frac{M_{all}(t_0)}{8  B_2}
\geq \left(\frac{1}{B_2}-\frac{1}{8B_2}\right)\frac{8B_2}{3}||n_{\al;\sim}(t_0+T_1+s)||_1,\quad \forall s\in[0,T_2].
\end{align}
Hence,
\begin{align}
\overline{n}_\al(t_0+T_1+s)\geq 2 ||n_{\al;\sim}(t_0+T_1+s)||_{L^1},\quad \forall \al\in \mathcal{I}, \forall s\in[0,T_2].
\end{align}
Now we can use the above information to estimate the time evolution of $M_{all}(t_0+T_1+s),\, \forall s\in[0, T_2]$ 
\begin{align}
\frac{d}{ds}&M_{all}(t_0+T_1+s)
=-\sum_{\al\in \mathcal{I}}\sum_{\beta\in\mathcal{I}}\ep_{\al\beta}\int n_\al n_\beta dV\\
=&-\sum_{\al\in \mathcal{I}}\sum_{\beta\in \mathcal{I}}\ep_{\al\beta} (\overline{n_\al}\, \overline{n_\beta}-\overline{n_{\al;\sim} n_{\beta;\sim}})\\
\leq&-\sum_{\al\in \mathcal{I}}\sum_{\beta\in \mathcal{I}}\ep_{\al\beta} \bigg(\overline{n_\al}(t_0+T_1+s)\, \overline{n_\beta}(t_0+T_1+s)-||n_{\al;\sim}^+(t_0+T_1+s)||_1||n_{\beta;\sim}^-(t_0+T_1+s)||_\infty\\
&-||n_{\beta;\sim}^+(t_0+T_1+s)||_1||n_{\al;\sim}^-(t_0+T_1+s)||_\infty\bigg).
\end{align}
Since $||n_{\al;\sim}^-||_\infty\leq \overline{n_\al}$, we have
\begin{align}\frac{d}{dt}&M_{all}(t_0+T_1+s)\\
\leq &-\sum_{\al\in \mathcal{I}}\sum_{\beta\in \mathcal{I}}\ep_{\al\beta} \bigg(\overline{n_\al}(t_0+T_1+s)\, \overline{n_\beta}(t_0+T_1+s)\\
&-\overline{n_\al}(t_0+T_1+s)\frac{1}{2}||n_{\beta;\sim}(t_0+T_1+s)||_1\, -\overline{n_\beta}(t_0+T_1+s)\frac{1}{2}||n_{\al;\sim}(t_0+T_1+s)||_1\bigg)\\
\leq&-\sum_{\al\in \mathcal{I}}\sum_{\beta\in \mathcal{I}}\ep_{\al\beta} (\overline{n_\al}(t_0+T_1+s)\, \overline{n_\beta}(t_0+T_1+s)/ 2).
\end{align}
Recalling that the assumption \eqref{Mass_lower_bound} holds on the time horizon $[t_0,t_0+T_\star]\subset [0,T]$, we have
\begin{align}
&\frac{d}{ds}M_{all}(t_0+T_1+s)\leq -\frac{1}{2 B_1}\min_\al\max_\beta \ep_{\al\beta}\sum_\al \overline{n_\al}(t_0+T_1+s)
= -\frac{1}{2 B_1}\left( \min_\al\max_\beta \ep_{\al\beta} \right) M_{all}(t_0+T+s).&
\end{align}
Now in time $T_2=2B_1(\min_\al\max_\beta \ep_{\al\beta})^{-1}\log\left(1-\frac{1}{8B_2}\right)^{-1}$,  sufficient mass is consumed, i.e.,
\begin{align}
M_{all}(t_0+T_\star)=M_{all}(t_0+T_1+T_2)\leq\bigg(1-\frac{1}{8B_2}\bigg)M_{all}(t_0).
\end{align}
This concludes the proof of \eqref{goal_system}. Hence the estimate \eqref{n_all_L_1_decay} follows.

\section{Proof of Theorem \ref{thm_3}} \label{Sec:thm_3}

In this section we prove Theorem \ref{thm_3}. The goal is to keep track of the total mass $||n_1(t)||_1$. To this end, we consider dynamics of the $x$-averages $\lan n_\al\ran$ \myr{defined in } \eqref{x_average_and_remainder} and design a $1D$-system to approximate their behaviors. By taking the $x$-average of the equations \eqref{Reaction-diffusion-shear}, we obtain
\begin{align}
\pa_t \lan n_\al\ran=\nu\pa_{yy}\lan n_\al\ran-\ep \lan n_\al\ran \lan n_\beta\ran-\ep\left\lan n_{\al;\neq}n_{\beta;\neq}\right\ran, \quad\beta\neq \al,\, \forall\al\in\{1,2\}.\label{zero_mode_equation}
\end{align}
To analyze the evolution of the system \eqref{zero_mode_equation},
we consider an intermediate one-dimensional periodic in space dynamics
\begin{align}
\pa_t \widetilde{n}_{1}=\nu\pa_{yy}& \widetilde{n}_{1}-\ep \widetilde{n}_{1}\widetilde{n}_{2},\quad
\pa_t \widetilde{n}_{2}=\nu\pa_{yy} \widetilde{n}_{2}-\ep \widetilde{n}_{1}\widetilde{n}_{2},\label{1D_dynamics}\\
(\widetilde{n}_{1}(t_0),&\widetilde{n}_{2}(t_0))=(\lan n_{1}\ran(t_0),\lan n_{2}\ran(t_0)).
\end{align} 


Before proving Theorem \ref{thm_3}, we present two lemmas. The first one provides an estimate of the $L^1$-distance between $\lan n_{\al}\ran$ and $\widetilde{n}_{\al}$. The other describes the evolution of the $L^1$-norms of the solutions $\wt n_\al$.
\begin{lem}\label{lem:distance-nal0-intermediate}
Consider the solutions $\lan n_{\al}\ran,\ \al \in \{1,2\}$ to \eqref{zero_mode_equation} and the solutions $\widetilde{n}_{\al}, \ \al\in\{1,2\}$ to the 1-dimensional dynamics \eqref{1D_dynamics}.
The $L^1$ distance between the two solutions are bounded in terms of the initial data as follows:
\begin{align}\label{Linfty_distance_bound_slow_reaction}
||\lan n_{\al}\ran-\widetilde{n}_{\al}||_{L_{y}^1}(t_0+t )- ||\lan n_{\al}\ran-\widetilde{n}_{\al}||_{L_{y}^1}(t_0 )\leq \ep \int_{t_0}^{t_0+t}\int|\lan n_{1;\neq}n_{2;\neq}\ran| dyds, \quad\al=1,2.
\end{align}
\end{lem}
\begin{proof}
The proof is based on the observation that the density differences $\lan n_1\ran-\lan n_2\ran$ and $\wt n_1-\wt n_2$ solve the same equation, i.e.,
\begin{align*}
\pa_t (\widetilde{n}_{1}-\widetilde{n}_{2})=&\nu\de (\widetilde{n}_{1}-\widetilde{n}_{2}),\\
\pa_t (\lan {n}_{1}\ran-\lan{n}_{2}\ran)=&\nu\de (\lan{n}_{1}\ran-\lan{n}_{2}\ran)
\end{align*}
and the two density differences 
share the same initial data. As a result, by uniqueness of heat equation, we have obtained the relation
\begin{align}\label{translation_relation}
\widetilde{n}_{1}(t,y)-\widetilde{n}_{2}(t,y)=\lan{n}_{1}\ran(t,y)-\lan{n}_{2}\ran(t,y), \, \forall t\in[0,\infty), \ \forall y\in\mathbb{T}.
\end{align}
Due to this relation \eqref{translation_relation}, we only need to estimate the $L^1$-distance between one component of the density difference. Without loss of generality, we consider $\lan n_1\ran -\wt n_1$
\begin{align*}
{\pa_t}&(\lan n_1\ran-\wt n_1)\\
=& \nu\de (\lan n_1\ran-\wt n_1)  -\ep(\lan n_1\ran\lan n_2\ran+\lan n_{1;\neq}n_{2;\neq}\ran-\wt n_1\wt n_2) \\
=&\nu\de (\lan n_1\ran-\wt n_1)-\ep \lan n_{1}\ran(\lan n_{2}\ran-\lan n_{1}\ran)+\ep \widetilde{n}_{1}(\widetilde{n}_{2}-\widetilde{n}_{1})-\ep \lan n_{1}\ran^2+\ep\widetilde{n}_{1}^2-\ep\lan n_{1,\neq}n_{2,\neq}\ran  .
\end{align*} 
Now we apply the relation \eqref{translation_relation} and rearrange the terms to obtain that
\begin{align*}
{\pa_t}(\lan n_1\ran&-\wt n_1) \\
=&\nu\de (\lan n_1\ran-\wt n_1)    -\ep (\lan n_2\ran -\lan n_1\ran+\lan n_1\ran +\wt n_1)(\lan n_1\ran-\wt n_1) -\ep\lan n_{1;\neq}n_{2;\neq}\ran\\
=&\nu\de (\lan n_1\ran-\wt n_1)  -\ep (\lan n_2\ran+\wt n_1)(\lan n_1\ran-\wt n_1) -\ep\lan n_{1;\neq}n_{2;\neq}\ran.
\end{align*}
We decompose the solution $\lan n_1\ran-\wt n_1$ as $f_1-f_2$, where the $f_i$'s solve the equations
\begin{align}
\pa_t f_1=&\nu \de f_1 -\ep (\lan n_2\ran+\wt n_1)f_1+\ep\lan n_{1;\neq}n_{2;\neq}\ran^{-},\quad f_1(t_0)=(\lan n_1\ran-\wt n_1)^+(t_0);\\
\pa_t f_2=&\nu \de f_2-\ep  (\lan n_2\ran+\wt n_1)f_2+\ep\lan n_{1;\neq}n_{2;\neq}\ran^{+},\quad f_2(t_0)=(\lan n_1\ran-\wt n_1)^-(t_0).
\end{align} By comparison principle, the $f_i$'s are non-negative. Hence $||\lan n_1\ran-\wt n_1||_1\leq ||f_1||_1+||f_2||_1$. Then integration of the $f_i$-equations in space and time yields the result. 
\end{proof}
 {\begin{lem}\label{lem:wt n_decay} The solutions $\wt n_\al$  to \eqref{1D_dynamics} have the following decay \myr{ for all $ t_0>0,t\geq 0$}:\quad
\begin{align*}
||&\min\{\wt n_1(t_0+t,\cdot), \wt n_2(t_0+t,\cdot)\}||_{L_y^1}\\
=& ||\min\{\wt n_1(t_0,\cdot), \wt n_2(t_0,\cdot)\}||_{L_y^1} -{\epsilon}\int_{t_0}^{t_0+t}\int\wt n_1(s,y) \wt n_2(s,y)dyds+\nu \int_{t_0}^{t_0+t}\sum_{y_i(s)}|\pa_y(\wt n_1-\wt n_2)(s,y_i(s))|ds.
\end{align*} 
Here the set $\{y_i(s)\}$ is the collection of points such that $\wt n_1(s,y_i(s))=\wt n_2(s,y_i(s))$ and  $\pa_y n_1(s,y_i(s))\neq \pa_y n_2(s,y_i(s))$.
\end{lem}
\begin{proof}
Recall that
\begin{align}
\min\{\wt n_1,\wt n_2\}=\frac{\wt n_1+\wt n_2}{2}-\frac{|\wt n_1-\wt n_2|}{2}.
\end{align}
Now we take the time derivative of the $L^1$-norm, and apply the observation that $\wt n_1-\wt n_2$ solves the heat equation to get
\begin{align}
\frac{d}{dt}||\min_{\al=1,2}\{n_\al\}(t_0+t,\cdot)||_1=&\int\frac{\pa}{\pa t}\left(\frac{\wt n_1+\wt n_2}{2}-\frac{|\wt n_1-\wt n_2|}{2}\right)dy\\
=&\int \left( \nu\pa_{yy} \left(\frac{\wt n_1+\wt n_2}{2}\right)-\ep \wt n_1 \wt n_2- \frac{(\wt n_1-\wt n_2)\pa_t(\wt n_1-\wt n_2)}{2|\wt n_1-\wt n_2|} \right) dy\\
=&-\ep \int\wt n_1 \wt n_2dy-\int\nu\frac{(\wt n_1-\wt n_2)\pa_{yy}(\wt n_1-\wt n_2)}{2|\wt n_1-\wt n_2|}dy.\label{ddt_min_L1}
\end{align}
The remaining part of the proof is to understand the last term in \eqref{ddt_min_L1}.

For the sake of notational simplicity, we use $\wt q$ to denote the  difference $\wt q:=\wt n_1-\wt n_2.$ The behavior of the last term in \eqref{ddt_min_L1} is related to the zero points of $\wt q$. Note that at the initial time $t_0$, $\wt n_1(t_0)-\wt n_2(t_0)=\lan n_1-n_2\ran (t_0), \,t_0>0$ and $\lan n_1-n_2\ran(t,y)$ as well as $\wt n_1- \wt n_2$ solve the heat equation on $\rr_+\times \Torus$. As a result, due to analyticity of solutions to heat equation, $\wt q(t_0+t)$ can only have finitely many zero points for $t_0>0,\,t\geq 0$. %
At any fixed instance, we label these finitely many zero points as  $\{y_i(t_0+t)\}_{i=1}^{N(t_0+t)}$ (zeros with multiplicities are labeled only once).  We further partition the torus $[-1/2,1/2]$ into $-\frac{1}{2}=y_0<y_1<y_2<y_3<...<y_{N}<y_{N+1}=\frac{1}{2}$ and define $I_i:=[y_i,y_{i+1}{})$. Note that $y_0=-\frac{1}{2}$ and $y_{N+1}=\frac{1}{2}$ are identified and { are not } the zero points. Since the solution $\wt q$ is smooth, the expression $\frac{\wt q}{|\wt q|}\pa_{yy}\wt q$ is smooth away from the points $\{y_i\}_{i=1}^N$. Moreover, the function $\wt q/|\wt q|$ is constant in the interior of $I_i$, i.e.,  $I_i^o:=(y_i,y_{i+1})$, so we denote it as $\left(\frac{\wt q}{|\wt q|}\right)(I_i^o)$. Also, since $y_0=y_{N+1}$ are 
not zeros in our set-up
 , $\left(\frac{\wt q}{|\wt q|}\right)(I_0^o)=\left(\frac{\wt q}{|\wt q|}\right)(I_{N}^o)$
. Combining the observations above, and the continuity of $\frac{\wt q}{|\wt q|}\pa_y\wt q$ at $y_0=y_{N+1}$ yields that 
\begin{align}
\int_{-1/2}^{1/2} \frac{\wt q}{|\wt q|}&\pa_{y y} \wt q dy=\sum_{i=0}^N\int_{I_i} \frac{\wt q}{|\wt q|}\pa_{y y} \wt q dy=\sum_{i=0}^ N \left(\frac{\wt q}{|\wt q|}\right)(I_i^o)\int_{y_i}^{y_{i+1}}\pa_{yy}\wt qdy \ \\
&=\sum_{i=0}^{N}\left(\frac{\wt q}{|\wt q|}\right) (I_i^o)\left(\pa_y\wt q(t_0+t,y)\bigg|_{y=y_i}^{y=y_{i+1}}\right)\\
&=\left(\frac{\wt q}{|\wt q|}\right)(I_0^o)\lim_{\eta\rightarrow 0+}\left(\pa_y \wt q\left(\frac{1}{2}-\eta\right)-\pa_y \wt q\left(-\frac{1}{2}+\eta\right)\right)\\
&\quad-\sum_{i=1}^{N }\lim_{\eta\rightarrow 0+}\left(\frac{\wt q(y_i+\eta)}{|\wt q(y_i+\eta)|}\pa_y\wt q(y_i+\eta)-\frac{\wt q(y_i-\eta)}{|\wt q(y_i-\eta)|}\pa_y \wt q(y_i-\eta)\right)\\
&=-\sum_{i=1}^{N } \pa_y \wt q(y_i) \lim_{\eta\rightarrow 0+}\left(\frac{\wt q(y_i+\eta)}{|\wt q(y_i+\eta)|}-\frac{\wt q(y_i-\eta)}{|\wt q(y_i-\eta)|}\right)\\
&=-\sum_{i=1}^{N(t_0+t)}2|\pa_y \wt q(t_0+t,y_i(t_0+t)) |.
\end{align}
Combining this calculation with \eqref{ddt_min_L1}, we have that
\begin{align}
{\frac{d}{dt}||\min_{\al=1,2}\{n_\al\}(t_0+t,\cdot)||_1=-\ep \int\wt n_1(t_0+t,y) \wt n_2(t_0+t,y)dy+\sum_{i=1}^{N(t_0+t)}\nu|\pa_y \wt q(t_0+t,y_i(t_0+t)) |.}
\end{align}
Integration in time yields the result.
\end{proof}
\begin{proof} [{Proof of Theorem \ref{thm_3}}]
We organize the proof in three steps.

\noindent
\textbf{Step \#1: Battle plan. }
First we define
\begin{align}
G:=\frac{1}{12}\norm{\min_{\al\in\{1,2\}}\lan n_{\al;0 }\ran}_{L_y^1}.
\end{align}
Hence the goal \eqref{comsumption_est_thm_3} can be interpreted as
\begin{align}\label{goal_shear}
||\lan n_{1;0}\ran||_{L_y^1}-||\lan n_1\ran(\mathcal T_0+\mathcal{T}_1)||_{L_y^1}\geq  G.
\end{align}
The strategy is similar to the one in Theorem \ref{thm_1}. Namely, we decompose the time horizon $[0,\mathcal{T}_0+\mathcal{T}_1]$ into two parts, i.e., $[0,\mathcal{T}_0)$ and $[\mathcal{T}_0,\mathcal{T}_0+\mathcal{T}_1]$. The enhanced dissipation estimates will be derived on the first interval and the reaction will be exploited on the second.

Next we make one further simplification. Same as before, we define the total reacted mass $Q(t)$, which is increasing in time,
\begin{align}
 Q(t):=\ep \int_0^{t}\int n_1n_2\, dxdy\,ds =||n_{1;0}||_{1}-|| n_{1}(t)||_1. 
\end{align}
Note that if there exists $t\in[0, \mathcal{T}_0+\mathcal{T}_1]$ such that $Q(t)\geq G$, then
\begin{align}
G\leq Q(t)\leq Q(\mathcal{T}_0+\mathcal{T
}_1)=||n_{1;0}||_1-||n_1(\mathcal T_0+\mathcal{T
}_1)||_1,
\end{align}which is the result \eqref{goal_shear}.
Therefore, it is enough to prove \eqref{goal_shear} under the assumption
\begin{align}
G \geq& Q(t),\quad \forall t\in[0,\mathcal T_0+\mathcal T_1].\label{Q_assumption}
\end{align}
This concludes step \# 1.

\noindent
\textbf{Step \# 2: Enhanced dissipation estimates on $[0,\mathcal{T}_0]$.}   Consider the solutions $\wh n_1,\, \wh n_2$ to the passive scalar equations\begin{align}
\pa_t\wh n_\al +u(y)\pa_x \wh n_\al=\nu \de \wh n_\al, \quad \wh n_\al(t=0,\cdot )=n_{\al;0}(\cdot),\quad \al\in\{1,2\}. 
\end{align}
The same argument as in the proof of \eqref{Q_difference} yields   that
\begin{align}\label{Dist_n_wh_n}
||n_{\al;0}||_{1}-|| n_{\al}(t)||_1=||\wh n_\al-n_\al||_1(t)=Q(t),\quad \al\in\{1,2\}.
\end{align}

Since the difference $q_{\neq}$ and  the  approximation $\wh n_{\al;\neq}$ solve the passive scalar equations, the enhanced dissipation  estimate \eqref{Defn:d_nu_shear_RE} applies, i.e., $||q_{\neq}(t)||_2\leq C ||q_{\neq}(0)||_2e^{-\delta d(\nu)t}$, $||\wh n_{\al;\neq}(t)||_2\leq C||\wh n_{\al;\neq}(0)||_2e^{-\delta d(\nu) t}$. By choosing the universal constant $C$ in the definition of $\mathcal{T}_0$ \eqref{t_0} large enough,  we have the following estimates at time $\mathcal T_0$,
\begin{align}\label{T_0_configuration}
 &|| q_{\neq}(\mathcal T_0)||_2\leq \frac{ \delta d(\nu)G}{121\ep\sum_{\al }||n_{\al;0}||_2};\ \
\sum_\al||\wh n_{\al;\neq}(\mathcal T_0)||_2\leq\frac{1}{121}{G}.
\end{align}
Moreover, on the time interval $[0,\mathcal T_0]$, we use Lemma \ref{lem:wt n_decay} with $\wt n_\al=\lan \wh n_\al \ran$, $\ep=0$ and $t_0=\mathcal T_0$ to obtain that
\begin{align}
||\min\{\lan\wh n_1\ran,\lan\wh n_2\ran\}(\mathcal T_0)||_{L^1_y}\geq ||\min\{\lan n_{1;0}\ran, \lan n_{2;0}\ran \}||_{L^1_y}=12G.
\end{align}
By recalling the relations
\begin{align}
\min\{\lan  n_1\ran,\lan  n_2\ran\}=\frac{\lan n_1\ran+\lan n_2\ran}{2}-\frac{|\lan n_1\ran-\lan n_2\ran|}{2},\quad \lan n_1\ran-\lan n_2\ran=\lan \wh n_1\ran-\lan \wh n_2\ran,
\end{align}
and combining them with \eqref{Dist_n_wh_n} 
and \eqref{Q_assumption}, we end up with
\begin{align}\label{L_1_0_and_T_0}
||\min\{\lan n_1\ran, \lan n_2\ran\}(\mathcal T_0)||_{L_y^1}\geq \frac{11}{12}||\min\{\lan n_{1;0}\ran, \lan n_{2;0}\ran \}||_{L_y^1}.
\end{align}
This concludes step \# 2.

\noindent
\textbf{Step \# 3: Reaction estimates on $[\mathcal{T}_0,\mathcal T_0+\mathcal T_1]$. }
On the second time interval, we compare $\lan n_1\ran$ to the solution $\wt n_1$ of the $1$D-system $\eqref{1D_dynamics}_{t_0=\mathcal T_0}$. To estimate their deviation, we first invoke the enhanced dissipation of $q_{\neq}$ \eqref{Defn:d_nu_shear_RE} and the estimate \eqref{T_0_configuration} to obtain
\begin{align*}
Q(\mathcal T_0+\mathcal T_1)-&Q(\mathcal T_0)=\ep\int_{\mathcal T_0}^{\mathcal T_0+\mathcal T_1}\int \lan n_1n_2\ran dydt\\
= &\ep\int_{\mathcal T_0}^{\mathcal T_0+\mathcal T_1}\int \lan n_1\ran \lan n_2\ran dydt+\ep\int_{\mathcal T_0}^{\mathcal T_0+\mathcal T_1} \int \lan q_{\neq}n_{2;\neq}\ran dydt+\ep\int_{\mathcal T_0}^{\mathcal T_0+\mathcal T_1}\int \lan n_{2;\neq}^2\ran dydt\\
\geq&\ep\int_{\mathcal T_0}^{\mathcal T_0+\mathcal T_1}\int \lan n_1\ran \lan n_2\ran dydt+\ep\int_{\mathcal T_0}^{\mathcal T_0+\mathcal T_1}\int \lan n_{2;\neq}^2\ran dydt-G/120. {}
\end{align*}
This estimate, when combined with Lemma \ref{lem:distance-nal0-intermediate}, yields the $L^1$-deviation control
\begin{align}
||\lan n_1\ran& -\wt n_1||_1(\mathcal T_0+t)\leq2\ep \int_{\mathcal T_0}^{\mathcal T_0+t}\int |\lan n_{1;\neq}n_{2;\neq}\ran| dyds \\
\leq&2\ep  \int_{\mathcal T_0}^{\mathcal T_0+t}\int| \lan n_{2;\neq}q_{\neq}\ran |d yds+2\ep\int_{\mathcal T_0}^{\mathcal T_0+t}\int\lan n_{2;\neq}^2\ran dyds\\
\leq &2Q(\mathcal T_0+t)-2Q(\mathcal T_0)+4G/120,\quad \forall t\in[0,\mathcal T_1=\mathcal{B}_1\ep^{-1}].\label{lan n_1 ran-wt_n_1_L1}
\end{align}

{Now we consider the total reacted mass associated with $1$D-system \eqref{1D_dynamics},
\begin{align}\label{I_t}
I(t)=\ep\int_{\mathcal T_0}^{\mathcal T_0+t}\int \wt n_1(s,y) \wt n_2(s,y)dyds.
\end{align}
 Recalling the 1-dimensional  equation \eqref{1D_dynamics} and direct $L^1$-estimate  yield the following relation
\begin{align}\label{I_t_relation}
||\wt n_{1}(\mathcal T_0)||_{L_y^1}-||\wt n_1(\mathcal T_0+t)||_{L_y^1}=||\wt n_{2}(\mathcal T_0)||_{L_y^1}-||\wt n_{2}(\mathcal T_0+t)||_{L_y^1}=I(t),\quad \forall t\in[0,\infty).
\end{align}
Therefore, given \eqref{lan n_1 ran-wt_n_1_L1}, to estimate the chemical consumed along the dynamics, it is enough to consider the time evolution of $I(t)$. By Lemma \ref{lem:wt n_decay}, we have that
\begin{align}
\int \min\{\wt n_1, \wt n_2\}(\mathcal T_0+t,y)dy+I(t)=
\int \min\{\wt n_{1}, \wt n_{2}\}(\mathcal T_0,y)dy+\int_{\mathcal T_0}^{\mathcal T_0+t}\nu \sum_{y_i}|\pa_y (\wt n_1(s,y_i)-\wt n_2(s,y_i))|ds.
\end{align}Here the $y_i$'s are specified in Lemma \ref{lem:wt n_decay}. 

We recall the definition of $\mathcal{B}_1$ in Theorem \ref{thm_3} and distinguish between two cases on the time interval $[\mathcal T_0,\mathcal T_0+\mathcal{B}_1\ep^{-1}]$.

\noindent
\textbf{Case a):} If there exists a constant $\mathcal{B}_2\in (0,\mathcal{B}_1]$
such that at time $\mathcal{B}_2 \ep^{-1}$, the following estimate holds
\begin{align}
\int \min\{\wt n_1, \wt n_2\}(\mathcal T_0+\mathcal{B}_2 \ep^{-1},y)dy\leq \frac{1}{2}\int \min\{\wt n_{1}, \wt n_{2}\}(\mathcal T_0,y)dy.
\end{align}
Then using \eqref{I_t_relation} and \eqref{L_1_0_and_T_0}, we obtain that
\begin{align}
I(\mathcal{B}_1\ep^{-1})\geq I(\mathcal{B}_2\ep^{-1})\geq \frac{1}{2}\int \min\{\wt n_{1}, \wt n_{2}\}(\mathcal T_0,y)dy\geq\frac{11}{2}G.
\end{align}
Hence by \eqref{I_t_relation}, we have a bound for the reacted total mass
\begin{align}
||\wt n_{1}(\mathcal T_0)||_{L_y^1}-||\wt n_1(\mathcal T_0+\mathcal{B}_1\ep^{-1})||_1\geq \frac{1}{2}\int \min\{\wt n_{1},\wt n_{2}\}(\mathcal T_0,y) dy\geq\frac{11}{2}G. 
\end{align}
Assumption \eqref{Q_assumption} and $L^1$-control \eqref{lan n_1 ran-wt_n_1_L1} yields that $||\lan n_1\ran(\mathcal T_0+t)-\wt n_1(\mathcal T_0+t)||_{L_y^1}\leq \frac{61}{30}G,\quad \forall t\in[0,\mathcal T_1]$. Hence, we have that
\begin{align}
||\lan n_{1;0}\ran||_{L_y^1}-||\lan n_1\ran(\mathcal T_0+\mathcal T_1)||_{L_y^1}\geq||\lan n_1\ran(\mathcal T_0)||_{L_y^1}-||\lan n_1\ran(\mathcal T_0+\mathcal{B}_1\ep^{-1})||_{L_y^1}\geq  G.
\end{align}
This concludes the proof in case a).

\noindent
\textbf{Case b):}  On the other hand, if on the time interval $[0, \mathcal{B}_1\ep^{-1}]$ the following estimate holds
\begin{align*}
\int \min\{\wt n_1, \wt n_2\}(\mathcal T_0+t,y)dy\geq \frac{1}{2}\int \min\{\wt n_{1}, \wt n_{2}\}(\mathcal T_0,y)dy,\quad \forall t\in [0, \mathcal{B}_1\ep^{-1}],
\end{align*}
then we can estimate $I(\mathcal{B}_1\ep^{-1})$ with H\"older's inequality as follows:
\begin{align}
I(\mathcal{B}_1\ep^{-1})\geq &\ep \int_0^{\mathcal{B}_1\ep^{-1}}\int_\Torus  \min\{\wt n_1, \wt n_2\}^2(\mathcal T_0+s,y) dyds\geq{\epsilon}\int_0^{\mathcal{B}_1\ep^{-1}}\left(\int_\Torus \min\{\wt n_1, \wt n_2\}(\mathcal T_0+s,y)dy\right)^2ds\\
\geq& \frac{\mathcal{B}_1}{4}\left(\int_\Torus \min\{\wt n_{1}, \wt n_{2}\}(\mathcal T_0,y)dy\right)^2.\label{I_T_0_to_T_1}
\end{align}
By \eqref{L_1_0_and_T_0}, we choose the universal constant $C$ in the definition of $\mathcal{B}_1$ \eqref{t_1} large enough so that
\begin{align}
\mathcal{B}_1\geq\max\{5,4||\min\{\lan n_1\ran,\lan n_2\ran\}(\mathcal T_0)||_{L^1_y}^{-1}\}=\max\{5,4||\min\{\wt n_1,\wt n_2\}(\mathcal T_0)||_{L_y^1}^{-1}\}.
\end{align}
Now if  $||\min\{\wt n_{1}, \wt n_{2}\}(\mathcal T_0)||_{L_y^1}\geq 1$, then  because $\mathcal{B}_1\geq 5$, the right hand side of \eqref{I_T_0_to_T_1} is greater than $\frac{1}{2} ||\min\{\wt n_{1}, \wt n_{2}\}(\mathcal T_0)||_{L_y^1}$ . If $0 <||\min\{\wt n_{1}, \wt n_{2}\}(\mathcal T_0)||_{L_y^1}\leq 1$,  the  choice $\dss \mathcal{B}_1\geq \frac{4}{||\min\{\wt n_{1}, \wt n_{2}\}(\mathcal T_0)||_{L^1_y}}$ yields the same lower bound as in the first case. To conclude, we have obtained the following estimate
\begin{align}
I(\mathcal{B}_1\ep^{-1})\geq\frac{1}{2}\int_\Torus \min\{\wt n_{1}, \wt n_{2}\}(\mathcal T_0,y)dy.
\end{align}
Now an application of the argument in case a) yields the result.
}\end{proof}
\appendix
\section{Appendix}
\subsection{Proof of the Enhanced Dissipation Estimate for Alternating Shear Flow}
In this section, we prove the enhanced dissipation estimate \eqref{ED_alternating_shear}. 

{We first consider the case $s,t\in 2K\nu^{-1/2}\mathbb{N}$ and comment on the general case at the end.  In this special case, the estimate \eqref{ED_alternating_shear} is guaranteed by the following
\begin{align}\label{goal_appendix_0}
||f_\sim(s+t)||_2\leq 2^{-\frac{t}{2K\nu^{-1/2}}}||f_\sim(s)||_2,\quad \forall s,t\in 2K\nu^{-1/2}\mathbb{N}.
\end{align}
For the sake of notation  simplicity, we drop the $(\cdot)_\sim$ notation in the appendix. Without loss of generality, we set $s=0$. Since the flow is time-periodic with period $2K\nu^{-1/2}$, it is enough to prove 
\begin{align}\label{goal_appendix_1}
||f(2K\nu^{-1/2})\leq \frac{1}{2}||f(0)||_2,
\end{align} 
given that $K,\,\nu^{-1}$ are chosen large enough. }
We decompose the interval $[0, 2K\nu^{-1/2}]$ into two parts
\begin{align}
[0, 2K\nu^{-1/2}]=[0,K\nu^{-1/2})\cup [K\nu^{-1/2},2K\nu^{-1/2}]. \label{Two_intervals_app}
\end{align}

On the interval $[0,K\nu^{-1/2})$
, the shear flow is given by 
\begin{align}
u(\tau,x,y)=\varphi_0(\tau)(\sin(2\pi y),0),
\end{align} where $\varphi_0$ is the $C^\infty$ time cut-off. We decompose the solution into the $x$-average and the $x$-remainder:
\begin{align}
f(\tau,x,y)=\lan f\ran_x(\tau,y)+f_{\neq_x}(\tau,x,y),\quad \lan f\ran_x(\tau,y)=\int_{-1/2}^{1/2} f(\tau,x,y)dx, \quad \iint f(\tau,x,y)dxdy=0.
\end{align}
Note that these two parts solve the \emph{decoupled} equations:
\begin{align}
\pa_\tau \lan f\ran_x=&\nu\pa_{yy} \lan f\ran_x,\quad \lan f\ran_x(0,y)=\lan f_0\ran_x(y);\\
\pa_\tau f_{\neq_x}+\varphi_0(\tau)\sin(2\pi y)\pa_x f_{\neq_x}=&\nu\de f_{\neq_x},\quad f_{\neq_x}(\tau=0,x,y)=(f_0)_{\neq_x}(x,y).
\end{align}
We focus on the remainder part. Recalling the enhanced dissipation estimate for shear flows $\eqref{ED_shear_introduction}_{j_m=1}$,  the non-expansive nature of the $ L^2$-norm  along the dynamics, and the fact that $\varphi_0(\tau)=1$ for $\forall \tau\in[\frac{1}{3}K\nu^{-1/2},\frac{2}{3}K\nu^{-1/2}]$, we have that for $\nu\in(0,\nu_0]$ and $K$ chosen large enough, 
\begin{align}
||f_{\neq_x}(K\nu^{-1/2})||_{L ^2}\leq& \norm{f_{\neq_x}\left(\frac{2}{3}K\nu^{-1/2}\right)}_{L ^2}
\leq C e^{-\delta \nu^{1/2}(K\nu^{-1/2}/3)}\norm{f_{\neq_x}\left(\frac{1}{3}K\nu^{-1/2}\right)}_{L ^2}\\
\leq&  \frac{1}{16}||f_{\neq_x}(0)||_{L^2}\leq\frac{1}{8}||f(0)||_{L^2}.\label{Time_1_est_app}
\end{align}
Here we take $0<\nu\leq \nu_0$ and $K\geq 3\delta^{-1}\log(16 C)$.

Now on the time interval $[K\nu^{-1/2},2K\nu^{-1/2}] $ in the decomposition \eqref{Two_intervals_app}, similarly to the previous argument, we decompose the solution into the following two parts
\begin{align}
\lan f\ran_y(\tau,x)=\int_{-1/2}^{1/2}  f(\tau,x,y)dy,\quad f_{\neq_y}(\tau,x,y)=f(\tau,x,y)-\lan f\ran_y(\tau,y).
\end{align}
Now the two quantities solve separate equations:
\begin{align}
&\pa_\tau \lan f\ran_y=\nu\pa_{xx} \lan f\ran_y;\ \
\pa_\tau f_{\neq_y}+\varphi_1(\tau)\sin(2\pi x)\pa_y f_{\neq_y}=\nu\de f_{\neq_y},
\end{align} which initiate at time $K\nu^{-1/2}$.
Note that due to the zero average constraint $\iint f(\tau,x,y) dxdy\equiv 0$, we have the relation $\lan \lan f\ran_x\ran_y=0$. Then as a consequence of \eqref{Time_1_est_app},  the non-increasing of the $L^2$-norm for solutions of the heat equation, and H\"older inequality we obtain that
\begin{align}
||\lan f\ran_y&(\tau)||_{L_{x}^2}\leq ||\lan f\ran_y(K\nu^{-1/2})||_{L_{x}^2}=||\lan \lan f\ran_x +f_{\neq_x}\ran_y (K\nu^{-1/2})||_{L_{x}^2}\\
=&||\lan f_{\neq_x}\ran_y(K\nu^{-1/2})||_{L_x^2} \leq || f_{\neq_x}(K\nu^{-1/2})||_{L^2_{x,y}}\leq \frac{1}{8}||f(0)||_{L^2_{x,y}},\quad \forall \tau\in[K\nu^{-1/2},2K\nu^{-1/2}].\label{rho_y_average_1}
\end{align}
Now for the remainder $f_{\neq_y}$, we use the enhanced dissipation estimate $\eqref{ED_shear_introduction}_{j_m=1}$, the non-expansive nature of the $L^2$ norm of solution, and the fact that $\varphi_1(\tau)=1,\ \forall{\tau}\in [4K\nu^{-1/2}/3, 5K\nu^{-1/2}/3]$ to obtain the following for $\nu\in(0,\nu_0],$ and $K\geq 3\delta^{-1}\log (16C)$: 
\begin{align}
 ||f_{\neq_y}( 2K\nu^{-1/2})||_{L_{x,y}^2}\leq Ce^{-\delta \nu^{1/2}(K\nu^{-1/2}/3)}||f_{\neq_y}(K\nu^{-1/2})||_{L_{x,y}^2}\leq \frac{1}{8}||f(0)||_{L_{x,y}^2}.&
\end{align}
Now combining the estimate with  \eqref{rho_y_average_1}, we obtain that
\begin{align}
||f(2K\nu^{-1/2})||_{L_{x,y}^2}\leq &||\lan f\ran_y(2K\nu^{-1/2})||_{L_{x,y}^2}+||f_{\neq_y}(2K\nu^{-1/2})||_{L_{x,y}^2}\leq\frac{1}{2}||f(0)||_{L_{x,y}^2}.
\end{align}
This concludes the proof of \eqref{goal_appendix_1} . 

For general $s,t \geq 0 $, we find the smallest integer $N$ and largest integer $M$ so that
\begin{align}
2KN\nu^{-1/2}\geq s,\quad 2KM\nu^{-1/2}\leq s+t,\quad M, N\in \mathbb{N}.
\end{align} Here $K$ is the same constant in the above analysis.
Note that if $t\leq 4K\nu^{-1/2}$, then the estimate \eqref{ED_alternating_shear} is direct:
\begin{align}
||f(s+t)||_2\leq ||f(s)||_2\leq 4||f(s)||_2 e^{-\frac{\log 2}{2K} \nu^{1/2}t},\quad 0\leq t\leq 4K\nu^{-1/2}.
\end{align}
Hence we assume $t>4K\nu^{-1/2}$ and observe that $2K\nu^{-1/2}(M-N)\geq t-4K\nu^{-1/2}$. Now we apply the estimate  \eqref{goal_appendix_0} with $s,t\in 2K\nu^{-1/2} \mathbb{N}$, and the non-increasing nature of $L^2$-norm of the solutions to derive that
\begin{align*}
||f(s+t)||_2\leq&||f(2KM\nu^{-1/2})||_2\leq ||f(2KN\nu^{-1/2})||_22^{-(M-N)}\leq ||f(s)||_2 e^{-\frac{\log 2}{2K}\nu^{1/2}2K\nu^{-1/2}(M-N)}\\%
\leq&||f(s)||_2 e^{-\frac{\log 2}{2K} \nu^{1/2} (t-4K\nu^{-1/2})}=4||f(s)||_2e^{-\frac{\log 2}{2K}\nu^{1/2} t},\quad \forall s,t\geq 0.
\end{align*}This concludes the proof of (\ref{ED_alternating_shear}) in the general case.

{\bf Acknowledgement}. \rm
The authors acknowledge partial support of the NSF-DMS grants 1848790, 2006372 and 2006660.
AK has been partially supported by Simons Fellowship. 
\myr{We would like to thank the anonymous referees and the associate editor for their thorough reading of the manuscript and many helpful suggestions.}
\ifx
\noindent
\textbf{Three Dimension:}
By the same type of argument, we might get the enhance dissipation of alternating shear flow in dimension three. The flow we will be considering
\begin{align}\label{alternating_shear_3D}
\mathbf{u}( t,x,y)=\sum_{k=0}^\infty&\varphi_{3k}(t)(\sin(y),0)+\sum_{k=0}^\infty\varphi_{3k+1}(t)(0, \sin(z),0)+\sum_{k=0}^\infty\varphi_{3k+2}(t)(0,0 ,\sin(x)),
\end{align}
where $\varphi_{\ell}$ is defined in \eqref{alternating_shear}:
\begin{align}
&\varphi_{\ell}(t)=\left\{\begin{array}{ccc}\ba 1,&\quad t\in [(\ell+1/3)C_\mathbf{u}\delta^{-1}\nu^{-1/2},(\ell+2/3)C_\mathbf{u}\delta^{-1}\nu^{-1/2}],\\
\mathrm{smooth},&\quad [\ell C_\mathbf{u} \delta^{-1}\nu^{-1/2}, (\ell+1/3)C_\mathbf{u}\delta^{-1}\nu^{-1/2}]\cup [(\ell+2/3)C_\mathbf{u}\delta^{-1}\nu^{-1/2},(\ell+1)C_\mathbf{u}\delta^{-1}\nu^{-1/2}),\\
0,&\quad \mathrm{others},\ea\end{array}\right.\\
&\varphi_\ell\in C_c^\infty,\quad \mathrm{support}( \varphi_\ell)\cap \mathrm{support}(\varphi_{\ell+1})= \emptyset,\quad\forall \ell\in\mathbb{N}.
\end{align}
Here $C_\mathbf{u}$ is a universal constant greater than $1$.

Same as in the previous argument, we decompose the time interval $[0,3C_\mathbf{u} \delta^{-1}\nu^{-1/2}]$ into three intervals:
\begin{align}
[0,3C_\mathbf{u} \delta^{-1}\nu^{-1/2}]=[0,C_\mathbf{u} \delta^{-1}\nu^{-1/2}]\cup [ C_\mathbf{u} \delta^{-1}\nu^{-1/2},{2}C_\mathbf{u} \delta^{-1}\nu^{-1/2}]\cup[2C_\mathbf{u} \delta^{-1}\nu^{-1/2},3C_\mathbf{u} \delta^{-1}\nu^{-1/2}]=:I_1\cup I_2\cup I_3.
\end{align}
On the first time interval $[0,C_\mathbf{u} \delta^{-1}\nu^{-1/2}]$, we decompose the solution into two parts:
\begin{align}
\lan f\ran_{x}(y,z)=\frac{1}{2\pi}\int_{-\pi}^\pi f(x,y,z)dx,\quad f_{\neq_x}(x,y,z)=f(x,y,z)-\lan f\ran_{x}(y,z).
\end{align}
Moreover, we define $\lan f\ran_y,\, \,\lan f\ran_z$ and $f_{\neq_y},\, \,f_{\neq_z}$ in the same fashion.
Similar to the argument in 2D,  the $x-$average and the remainder solve separate equations and the remainder part undergoes enhanced dissipation. Now at the end of the time interval $I_1$, we have that (The physical variables $x,y,z$ correspond to the Fourier variables $k,\ell, j$, respectively. We absorb the universal constants $2\pi$ of the Fourier transform in the constant $C_{ED}$.)
\begin{align}
\sum_{k\neq 0}|\whf_{k,0, 0}(C_\mathbf{u} \delta^{-1}\nu^{-1/2})|^2\leq C_{ED}e^{-\delta \nu^{1/2}(C_\mathbf{u} \delta^{-1}\nu^{-1/2})/3}||f(0)||_2^2\leq \frac{1}{256}||f(0)||_2^2.
\end{align}
Moreover, since the expression $\sum_{k\neq0}e^{2\pi i k x}\whf_{k,0,0}(t)$ is in the kernel of the differential operators $\sin(z)\pa_y$ and $\sin(x)\pa_z$, we apply the dissipative nature of the equation to obtain the following
\begin{align}
\sum_{k\neq 0}|\whf_{k,0, 0}(t)|^2\leq \frac{1}{256}||f(0)||_2^2,\quad \forall t\in [C_\mathbf{u} \delta^{-1}\nu^{-1/2}, 3C_\mathbf{u} \delta^{-1}\nu^{-1/2}].\label{piece_1}
\end{align}
On the time interval $I_2$, we spit the solution into the $y-$average $\lan f\ran_y$ and the remainder $f_{\neq_y}$. One observes enhanced dissipation in the remainder $f_{\neq_y}$. As a result, we have the following estimate
\begin{align}
\sum_{\ell\neq 0}\sum_{k\in \mathbb{Z}}|\wh f_{k,\ell,0}(2C_{\mathbf{u}}\delta^{-1}\nu^{-1/2})|^2\leq C_{ED}e^{-\delta\nu^{1/2}(C_{\mathbf{u}}\delta^{-1}\nu^{-1/2})/3}||f(0)||_2^2\leq\frac{1}{256}||f(0)||_2^2.\label{piece_2}
\end{align}
Combining the estimates \eqref{piece_1}, \eqref{piece_2} and the zero average condition $\wh f_{0,0,0}\equiv 0$, we have that the following estimate on the $z-$average holds:
\begin{align}
||\lan f\ran_z(2C_\mathbf{u}\delta^{-1}\nu^{-1/2})||_2^2=C_{Fourier}\sum_{(k,\ell)\in \mathbb{Z}^2}|\whf_{k,\ell,0}(C_\mathbf{u}\delta^{-1}\nu^{-1/2})|^2\leq \frac{C_{Fourier}}{128}||f(0)||_2^2.
\end{align}
Now observe that the $z-$average solves the heat equation on the time interval $I_3$, the following estimate holds:
\begin{align}
||\lan f\ran_z(3C_\mathbf{u}\delta^{-1}\nu^{-1/2})||_2^2\leq \frac{C_{Fourier}}{128}||f(0)||_2^2.
\end{align}
Since the $z-$remainder $f_{\neq_z}$ undergoes enhanced dissipation, we have that
\begin{align}
||f_{\neq_z}(3C_\mathbf{u}\delta^{-1}\nu^{-1/2})||_2
\leq C_{ED}e^{-\delta\nu^{1/2}(C_\mathbf{u} \delta^{-1}\nu^{-1/2})/3}||f(0)||_2\leq \frac{1}{64}||f(0)||_2.
\end{align}
Now combining the estimate on $\lan f\ran_z$ and $f_{\neq_z}$, we obtain that
\begin{align}
||f(3C_{\mathbf{u}}\delta^{-1}\nu^{-1/2})||_2\leq\frac{1}{2}||f(0)||_2.
\end{align}
This concludes the proof of the enhanced dissipation in 3-dimension.

This method of constructing the alternating shear flows gives concrete examples of smooth fluid flows  which are powerful enough to do the dimension reduction in various nonlinear models.
\fi

\bibliographystyle{abbrv}
\bibliography{nonlocal_eqns,JacobBib,SimingBib}

\end{document}